\documentclass[final,leqno]{scrarticle}

\usepackage[table]{xcolor}
\usepackage{mathtools}
\usepackage{todonotes}

\usepackage{amsmath}
\usepackage{amssymb}
\usepackage{amsthm}
\usepackage{float}
\usepackage{comment}
\usepackage{graphicx}
\usepackage{textcomp}
\usepackage{appendix}
\usepackage{authblk}
\usepackage{subcaption}
\usepackage{srcltx} 
\usepackage{hyperref}
\usepackage{url}
\hypersetup{draft=false, colorlinks=true, linkcolor=black, urlcolor=blue, citecolor=black}
\makeatletter
\g@addto@macro{\UrlBreaks}{\do\/\do-\do\_} 
\makeatother
\newtheorem{theorem}{Theorem}[section]
\newtheorem{lemma}[theorem]{Lemma}
\newtheorem{remark}[theorem]{Remark}

\newtheorem{definition}[theorem]{Definition}

\DeclareMathOperator{\adj}{adj}
\DeclareMathOperator{\dx}{dx}
\DeclareMathOperator{\ds}{ds}

\DeclareMathOperator{\diver}{div}
\DeclareMathOperator{\tr}{tr}
\newcommand{\R}{\mathbb{R}} 
\newcommand{\N}{\mathbb{N}} 

\newcommand{\M}{\mathbb{M}}
\newcommand{\X}{\mathbb{X}}

\newcommand{\euler}{\mathrm{e}}

\providecommand{\keywords}[1]
{
  {\small	
  \textbf{{Keywords---}} #1}
}

\author[1]{Gesa Sarnighausen}
\author[2]{Tram Thi Ngoc Nguyen}
\author[3]{Thorsten Hohage}
\author[4]{Mangalika Sinha}
\author[5]{Sarah K\"oster}
\author[6]{Timo Betz}
\author[7]{Ulrich Sebastian Schwarz}
\author[8]{Anne Wald}
\affil[1]{Institute for Numerical and Applied Mathematics, University of G\"ottingen, Germany (\texttt{g.sarnighausen@math.uni-goettingen.de})}
\affil[2]{Max Planck Institute for Solar System Research, G\"ottingen, Germany (\texttt{nguyen@mps.mpg.de})}
\affil[3]{Institute for Numerical and Applied Mathematics, University of G\"ottingen, Germany (\texttt{t.hohage@math.uni-goettingen.de})}
\affil[4]{Institute for X-Ray Physics, University of G\"ottingen, Germany (\texttt{mangalika.sinha@uni-goettingen.de})}
\affil[5]{Institute for X-Ray Physics, University of G\"ottingen, Germany (\texttt{sarah.koester@phys.uni-goettingen.de})}
\affil[6]{Third Institute of Physics, Biophysics/ Complex Systems, University of G\"ottingen, Germany (\texttt{timo.betz@phys.uni-goettingen.de})}
\affil[7]{Institute for Theoretical Physics, Heidelberg University, Germany (\texttt{schwarz@thphys.uni-heidelberg.de})}
\affil[8]{Institute for Numerical and Applied Mathematics, University of G\"ottingen, Germany (\texttt{a.wald@math.uni-goettingen.de})}

\title{Traction force microscopy for linear and nonlinear elastic materials as a parameter identification inverse problem}
        
\begin{document}

\maketitle

\begin{abstract}
\noindent
Traction force microscopy is a method widely used in biophysics and cell biology to determine forces that biological cells apply to their environment. In the experiment, the cells adhere to a soft elastic substrate, which is then deformed in response to cellular traction forces. The inverse problem consists in computing the traction stress applied by the cell from microscopy measurements of the substrate deformations. In this work, we consider a linear model, in which 3D forces are applied at a 2D interface, called 2.5D traction force microscopy, and a nonlinear pure 2D model, from which we directly obtain a linear pure 2D model. All models lead to a linear resp.~nonlinear parameter identification problem for a boundary value problem of elasticity. We analyze the respective forward operators and conclude with some numerical experiments for simulated and experimental data.
\end{abstract}



\section{Introduction}


Cells are the fundamental building blocks of all living systems. They consist of a multitude of biomolecules that continuously interact with each other and thereby form distinct structures. In particular, a large number of proteins are involved in building the cytoskeleton, which is a system of different types of filaments that together give mechanical stability
to cells, which they need to move, divide and form tissue \cite{lorenz2022multiscale}. Biological systems are active in the sense that they continuously consume energy and thus keep the system away from thermodynamic equilibrium. This not only
allows them to form structures that are not possible in equilibrium, it also allows them to quickly
adapt to changes in their environment \cite{ahmed2015active,banerjee2020actin}. A particular point in case
are molecular motors, which are large proteins that use energy to undergo conformational changes that lead
to force generation. Together with other force-generating processes like polymerization, 
this activity allows cells to change their shape, migrate, divide, and adhere to surfaces
and other cells. Because it is difficult to measure cell forces inside cells, it has become
a standard procedure in biophysics and cell biology to measure them at the interface between
the cells and their environment, with a procedure called traction force microscopy (TFM) \cite{TFMPhysicsBiologyStyle,Schwarz2015,Lekka2021,zancla2022primer,reviewDenisin}.

In TFM, cells are cultured on soft elastic substrates. Fiducial markers, i.e., small fluorescent beads, are embedded in the substrate and the displacement of the substrate is calculated by comparing observed images of the markers first with and again without applied traction stresses (compare Figure~\ref{fig:tfmsetting}). The inverse problem then consists in computing the traction stresses given the displacement, which are related via the boundary value problem (BVP) of elasticity for the substrate. Depending on the chosen material law the BVP can be linear or nonlinear. Hence, we refer to linear or nonlinear TFM.

When cells are plated on a planar substrate coated with proteins from the extracellular matrix, they spread out to increase surface contact with the substrate and to establish stable adhesions. Thus, they tend to thin out, as a result of which the cellular forces are predominantly in-plane. This effect is further enhanced by the fact that for 
adhesion to a planar substrate, the cytoskeleton typically flows parallel to the substrate, from 
the periphery to the center, thus also generating mainly tangential (and centrally directed)
forces. This explains the name 'traction force', i.e., the force acting tangentially on the substrate plane,
similar to a car on the road. However, in general, cellular traction forces need not be tangential,
because several aspects are leading also to forces in the perpendicular direction, including
the large nucleus sticking out into the third dimension, cortical tension pulling everywhere in the cell
envelope, and the internal organization of the adhesion sites \cite{TFMPhysicsBiologyStyle,reviewDenisin}.
For example, cancer cells often try to invade the substrate, and then sometimes
use their nucleus to generate perpendicular forces onto the substrate \cite{kristal2013metastatic}.

The output of TFM typically is a map of traction stress measured in physical units of Pa=N/m$^2$.
A typical value for cellular traction forces is kPa, which makes sense because this is
also the typical value of the stiffness of an extracellular matrix. Only if these two 
numbers are of similar magnitude, the cell and its environment are at a reasonable
mechanical balance to each other \cite{schwarz2017mechanobiology}. Although from the viewpoint
of elasticity theory, one measures stress and not force, this technique is called traction force microscopy because the whole notion of this technique lies in understanding the forces exerted by the cells on the substrate.  Both the term 'traction force' and 'traction stress' are thus complementary to each other since they provide a complete picture of cell mechanics, demonstrating both the magnitude and distribution of forces.

\begin{figure}[h]
    \centering
    \begin{subfigure}{0.75\textwidth}
    \centering
    \def\svgwidth{\linewidth} 
\begingroup%
  \makeatletter%
  \providecommand\color[2][]{%
    \errmessage{(Inkscape) Color is used for the text in Inkscape, but the package 'color.sty' is not loaded}%
    \renewcommand\color[2][]{}%
  }%
  \providecommand\transparent[1]{%
    \errmessage{(Inkscape) Transparency is used (non-zero) for the text in Inkscape, but the package 'transparent.sty' is not loaded}%
    \renewcommand\transparent[1]{}%
  }%
  \providecommand\rotatebox[2]{#2}%
  \newcommand*\fsize{\dimexpr\f@size pt\relax}%
  \newcommand*\lineheight[1]{\fontsize{\fsize}{#1\fsize}\selectfont}%
  \ifx\svgwidth\undefined%
    \setlength{\unitlength}{512.74287667bp}%
    \ifx\svgscale\undefined%
      \relax%
    \else%
      \setlength{\unitlength}{\unitlength * \real{\svgscale}}%
    \fi%
  \else%
    \setlength{\unitlength}{\svgwidth}%
  \fi%
  \global\let\svgwidth\undefined%
  \global\let\svgscale\undefined%
  \makeatother%
  \begin{picture}(1,0.22073141)%
    \lineheight{1}%
    \setlength\tabcolsep{0pt}%
    \put(0,0){\includegraphics[width=\unitlength,page=1]{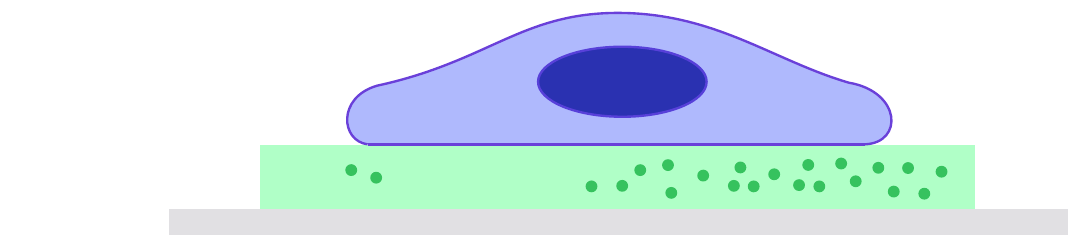}}%
    \put(0.04492841,0.0093464){\color[rgb]{0,0,0}\makebox(0,0)[lt]{\lineheight{1.25}\smash{\begin{tabular}[t]{l}glass\end{tabular}}}}%
    \put(0.04628064,0.04952529){\makebox(0,0)[lt]{\lineheight{1.25}\smash{\begin{tabular}[t]{l}substrate\end{tabular}}}}%
    \put(0.31728025,0.19998664){\makebox(0,0)[lt]{\lineheight{1.25}\smash{\begin{tabular}[t]{l}cell\end{tabular}}}}%
    \put(0.03623911,0.12690158){\makebox(0,0)[lt]{\lineheight{1.25}\smash{\begin{tabular}[t]{l}marker beads\end{tabular}}}}%
    \put(0,0){\includegraphics[width=\unitlength,page=2]{setup.pdf}}%
  \end{picture}%
\endgroup%

    \caption{Experimental setup for TFM.}
    \end{subfigure}
    \\
    \hfill
    \begin{subfigure}[T]{0.4\textwidth}
    \centering
    \includegraphics[width=\linewidth]{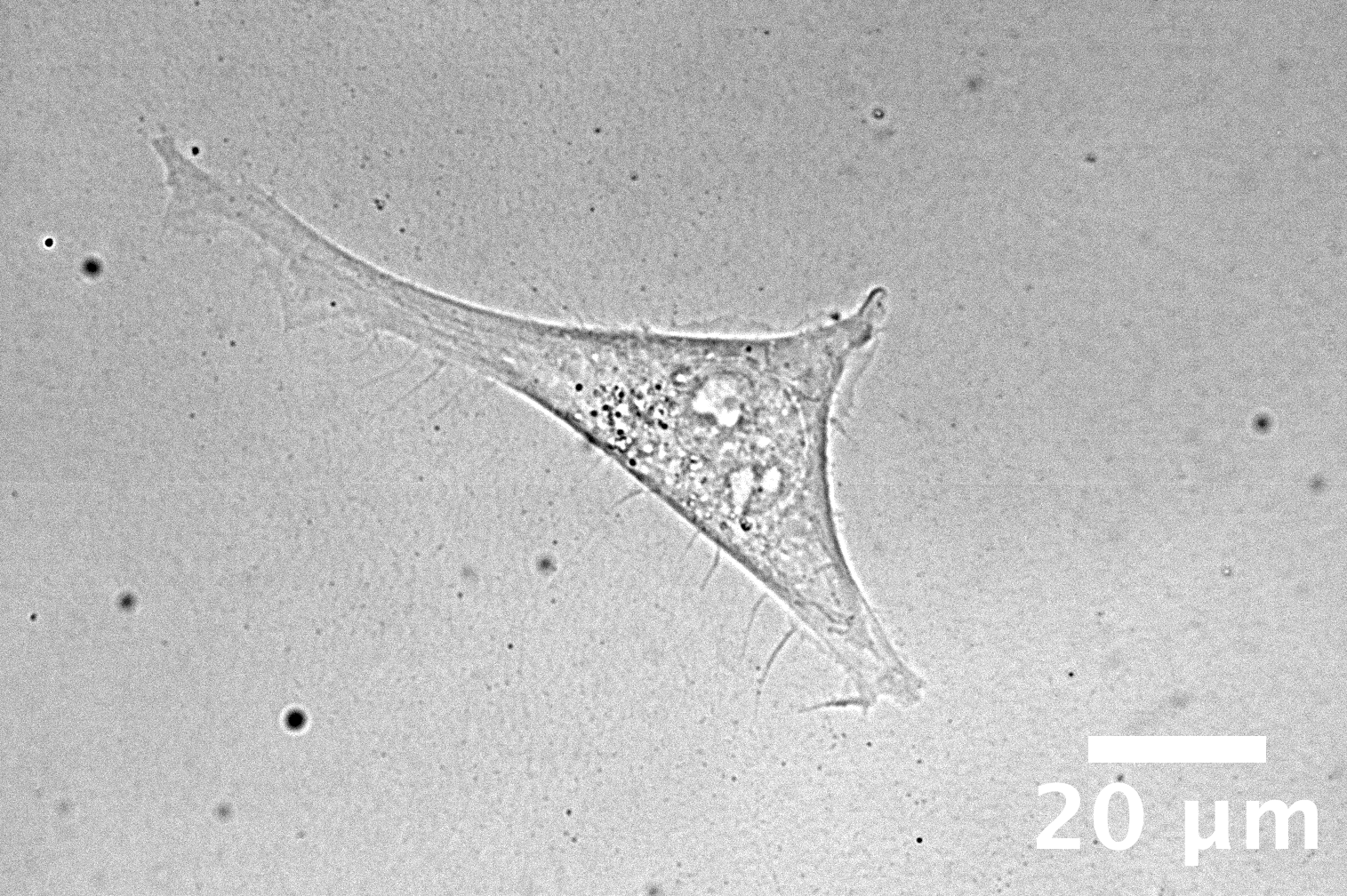}
    \vspace{4ex}
    \caption{Microscopy image of a cell.}
    \label{fig:microscopyCellImage}
    \end{subfigure} \hfill
    \begin{subfigure}[T]{0.45\textwidth}
    \centering
    \includegraphics[width=\linewidth]{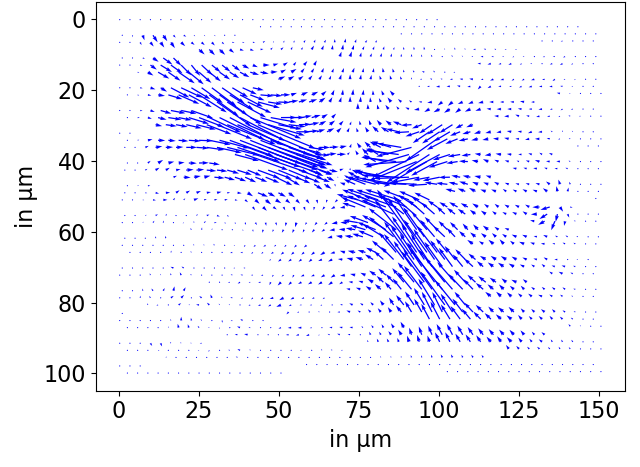}
    \caption{Image of the deformation.}
    \label{fig:microscopyCellDeformation}
    \end{subfigure}
    \caption{(a) Experimental setup for TFM. Small fiducial markers (dark green) are embedded in a soft elastic substrate (light green). The cell (blue; nucleus in dark blue) adheres to the substrate and applies traction stresses. For 2D TFM one often only uses markers in a thin layer just underneath the substrate surface, as only tangential deformations are being considered. For 2.5D TFM the markers have to be distributed in the whole substrate, because one also has to track perpendicular deformations. In practice, the substrate is much thicker than the force generating structures in the cell (about $50$ $\mu$m vs. a few $\mu$m). The markers are typically imaged with an inverse optical microscope from below. \\
    (b) Microscopy image of a cell plated on a planar substrate.
    (c) Image of the corresponding deformation.}
    \label{fig:tfmsetting}
\end{figure}

Traction force observations were first developed in 1980 by Harris et al., who were the first to culture cells on a substrate and observe the traction stresses of the cells by wrinkling and elastic distortion of the substrate \cite{Harris1980, Harris1981}. Over 10 years later Dembo et al.~developed the first TFM method by finding a way to quantitatively compute the applied stresses from the measured displacement, first for thin films \cite{Dembo1996} 
and then for thick ones \cite{Dembo1999}. For thick films, they considered a linear 2D model, ignoring the normal component, and then used Green's functions from the analytical Boussinesq solution for the elastic halfspace, i.e., they assume an infinite substrate with homogeneous Dirichlet boundary conditions at infinity, leading to an ill-posed inverse problem, and solved it in real space.
It is also possible to solve the inverse problem in Fourier space without regularization \cite{butler2002}. Since the inverse problem is ill-posed, noisy data can lead to severe artifacts in the reconstruction if no regularization is used \cite{Schwarz2002}. Therefore regularization was also incorporated into the Fourier space setting \cite{Sabass2008,Plotnikov2014}. This method is now mostly used in practice to reconstruct traction stresses. Note, that even if this approach is called 2D TFM, due to the use of the Boussinesq solution the model is not purely 2D.

If the substrate cannot be modeled by a linear material law, the standard approach does not work anymore since there exists no analytical Boussinesq solution for the underlying BVP. In this case, finite element methods can be used, e.g., by solving a PDE-constrained optimization problem as in Ambrosi et al.~\cite{Ambrosi2006,Ambrosi2009}. Michel et al. have proposed a mathematical framework both for Ambrosi's solution approach and the standard Fourier method~\cite{Michel2013}.
It is further possible to solve the BVP of elasticity with respect to the stress $\sigma$ and then directly compute the traction stresses $t$ by the relationship $t = \sigma n$, where $n$ is the normal vector \cite{Hur2009,Hur2012}.

TFM can also be formulated as a direct problem, by calculating stress directly from strain \cite{Stacey2009,Franck2007,BarKochba2014,schwarzComparison}. However, high-quality and three-dimensional measurements are necessary because the direct method includes taking derivatives of the noisy measured data. This approach is often used in 2.5D and 3D due to the lack of an analytical solution for the respective BVP.

Another strategy is model-based TFM, when traction forces are not reconstructed at the 
substrate surface, but directly fitted against a physical model of force generation.
While this makes the results depend on the chosen model, it allows us to
gain information about how the traction stresses are distributed in the cell by incorporating modeling assumptions on the cells \cite{Tambe2010,Tambe2013,Ng2014,Soine2015}. Similar to direct TFM, high quality measurements are necessary to obtain reasonable results. 

During the last three decades since its invention, TFM has matured into a large research area, 
and many different experimental setups and solution strategies have been devised. To
systematize the field, it is helpful to make a few distinctions. One
important distinction is between direct and inverse approaches; while in the direct method
stresses are directly calculated from strains, in the inverse method, stresses are estimated
as minimizers of the differences between predicted and observed displacements \cite{schwarzComparison}.
Another helpful distinction can be made
between 2D, 2.5D, compare
Figure~\ref{fig:tfmsetting}, and 3D TFM \cite{Schwarz2015}. The 2D setting is the simplest. A cell is placed onto a planar soft elastic substrate with marker beads
and one follows their deformations. By ignoring the normal component of cell forces and displacement, it is assumed that the cell forces are just two-dimensional. 3D TFM refers to a cell completely embedded inside a 3D matrix, e.g.~hydrogel gel, with beads all around it. In principle, it is then possible to determine three-dimensional traction forces around the whole cell. However, the imaging of the displacement is very difficult and time-consuming in 3D. More importantly, the material law for 3D material allowing for 3D TFM is typically very complex and often not even elastic. Therefore a good
compromise is 2.5D TFM. As in 2D TFM, a cell is placed on a planar substrate as shown in Figure~\ref{fig:tfmsetting}. However, marker beads are not only tracked close to the surface but throughout the whole substrate, yielding a three-dimensional displacement. In this article, we will consider both pure 2D and 2.5D TFM.

Even though TFM is widely used in biophysics and cell biology, the mathematical literature on TFM is scarce~\cite{Ambrosi2006,Michel2013}. This article aims to develop a rigorous mathematical theory in a function space setting for both linear and nonlinear TFM and to solve the inverse problem efficiently by regularization methods. We analyze a continuous model and only discretize for implementation, while physics approaches often discretize already
the initial model. In linear/nonlinear TFM a linear/nonlinear material as a substrate and the corresponding linear/nonlinear strain tensor are used leading to a linear/nonlinear inverse problem, respectively. Most materials are linear at small deformations, which can be ensured by using a material that is a bit stiffer than typical traction forces. 
However, stiff materials typically change cell behavior, and thus one often also uses soft materials and
large deformations, which might require nonlinear material laws.

We first solve the linear inverse problem of 2.5D TFM with a three-dimensional substrate using a similar method as Soin\'e et al.~\cite{soinePhd}. With our functional analytic setting, we can easily apply different standard algorithms from the theory of inverse problems.
We then solve the nonlinear inverse problem in a pure 2D setting for the substrate's surface and a hyperelastic material law. Hyperelasticity is a mathematical concept as introduced in Section \ref{sec:elasticity}. One method used in physics for solving nonlinear TFM is the direct method \cite{Franck2007,barrasa2021advanced,schwarzComparison}, i.e., directly computing traction forces from the measured data by numerically differentiating the noisy data. As in the linear 2.5D case, we propose a model formulation that can be solved in a straightforward manner using different regularization methods. From this approach, we directly derive a solution method for linear pure 2D TFM which is used for experimental data in Section~\ref{ssec:linearRealData}. We can use this approach either for computing the traction if the effective thickness is known or for estimating the effective thickness beyond which the horizontal displacements disappear. As described above, this approach differs from the standard approach in physics because the use of the Boussinesq solution makes the model not purely 2D, whereas we consider a pure 2D model.

In summary, this article provides a rigorous mathematical analysis for linear and nonlinear pure 2D TFM as well as 2.5D linear TFM along with all necessary tools to solve these inverse problems numerically in a stable way. It can be seen as a mathematical basis to allow for nonlinear substrate materials.

\emph{Outline.} In this article, we propose a TFM model leading to a parameter identification problem that can be solved using functional analytic tools and regularization methods from inverse problems theory. We propose a linear 2.5D model and a nonlinear pure 2D model which automatically leads to an analogous linear pure 2D model as well.
The article is structured as follows. The first part deals with the linear 2.5D model, see Section~\ref{ch:linTFM}. After presenting the basics of elasticity theory in Section~\ref{sec:elasticity}, the linear mathematical model is derived and analyzed in Section~\ref{sec:modelTFM}.
In the second part (Section~\ref{ch2:nonlinearTFM}), we first derive the nonlinear model in Section~\ref{sec2:modelNonlin}. Then the well-posedness of the forward operator is investigated, see Section~\ref{sec2:well-posedness}, and the Fr\'echet derivative and its adjoint are computed in Section~\ref{sec2:Frechet}. The third part, Section~\ref{sec:numerics}, contains various numerical experiments for both simulated and experimentally measured data.


\section{Linear 2.5D traction force microscopy}
\label{ch:linTFM}


If the substrate is modeled by a linear material law, the traction stress $t$ and the displacement $u$ are connected by the linear displacement-traction BVP of elasticity.
With a forward operator that maps a traction stress $t$ to the corresponding displacement $u$, we can consider TFM as a linear parameter identification problem.

\subsection{Basics of elasticity}
\label{sec:elasticity}
Let $n \in \{2,3\}$.
For a deformation $\varphi: \R^n \to \R^n$, we define the displacement vector 
$
    u(x) \coloneqq \varphi(x) - x.
$
Using this definition, we can express the deformation gradient $F \coloneqq \nabla \varphi(x) \in \R^{n\times n}$ as
\begin{equation} \label{eq:Fu}
    F(t,x) = \nabla u(x) + I \in \R^{n \times n},
\end{equation}
where $I$ is the identity matrix.
The deformation gradient is used to define a nonlinear strain tensor, the Green-Lagrange tensor
\begin{equation}
\label{eq1:nonlinearstrain}
    E \coloneqq \frac{1}{2} \left( F^{\top} F - I \right) \in \R^{n \times n}.
\end{equation}

The Green-Lagrange tensor $E$ is often represented as
\begin{equation}
\label{eq:Green-Lagrange}
    E = \frac{1}{2} \left( \nabla u + (\nabla u)^{\top} + (\nabla u)^{\top}(\nabla u) \right) \in \R^{n \times n}
\end{equation}
by means of \eqref{eq:Fu}. 
For small strains $\varepsilon$ resp.~deformations, one often uses the linearized Green-Lagrange tensor
\begin{equation}
    \varepsilon \coloneqq \frac{1}{2} \left( \nabla u + (\nabla u)^{\top} \right) \in \R^{n \times n},
\end{equation}
omitting the high-order term.
The relation between the strain tensor $\varepsilon$ or $E$ and the stress tensor $\sigma$ 
is called a constitutive equation and depends on the material.

If the substrate material is linear elastic and isotropic, we have a linear constitutive relation between $\varepsilon$ and $\sigma$, given by Hooke's law
\begin{align}
\label{eq:Hookeslaw}
\sigma(u) = \lambda \diver(u) I + 2 \mu \varepsilon \in \R^{n \times n}
\end{align}
in terms of the Lam\'e constants $\lambda,\mu \in \R $, related to the Young modulus $\tilde{E}$ and the Poisson ratio $\nu$ via 
\begin{align}
\label{eq:lame}
\mu = \frac{\tilde{E}}{2(1+\nu)} > 0, \quad \lambda = \frac{\tilde{E} \nu}{(1 + \nu)(1-2\nu)},
\end{align} 
see \cite[p. 128]{CiarletI}.
The first Lam\'e parameter $\lambda$ is positive for a positive Poisson ratio $\nu < \frac{1}{2}$. Here, we only consider this case.

A material is called hyperelastic if there exists a function $W(F)$, called the stored energy function, such that
\begin{align}
\label{eq:hyperelastic}
    \sigma (x,F) = \frac{\partial W}{\partial F} (x,F).
\end{align}
Hyperelastic constitutive equations model nonlinear material behavior as well as large shape changes. Examples of constitutive equations and their physical background can be found in~\cite[chapter 3.5]{bower2009} or, with a stronger emphasis on mathematics, in~\cite{botti2022}.
More details about elasticity can be found, e.g., in~\cite{CiarletI, holzapfel2000} or \cite[chapter 61/62]{Zeidler1997NonlinearFA4}.

\subsection{A mathematical model of linear 2.5D TFM}
\label{sec:modelTFM}

Consider a non-deformed region $\emptyset \neq \Omega \subset \R^3$ which is bounded, open and connected with a smooth boundary $\partial \Omega \in C^{0,1}$ and $\partial \Omega = \overline{\Gamma_D} \cup \overline{\Gamma_{T^*}}$ for disjoint, relatively open subsets $\Gamma_D, \Gamma_{T^*} \subset \partial \Omega$ and $\Gamma_D \neq \emptyset$. In the case of TFM the region of interest is part of the substrate which is assumed to be a cuboid, see Figures~\ref{fig:tfmsetting} and~\ref{fig:substrate}. The bottom surface of the cuboid is called $\Gamma_D$, the side surface $\Gamma_{T_0}$ and the top surface $\Gamma_T$ with $\Gamma_{T^*} = \Gamma_T \cup \Gamma_{T_0}$

\begin{figure}
    \centering
    \includegraphics[width=8cm]{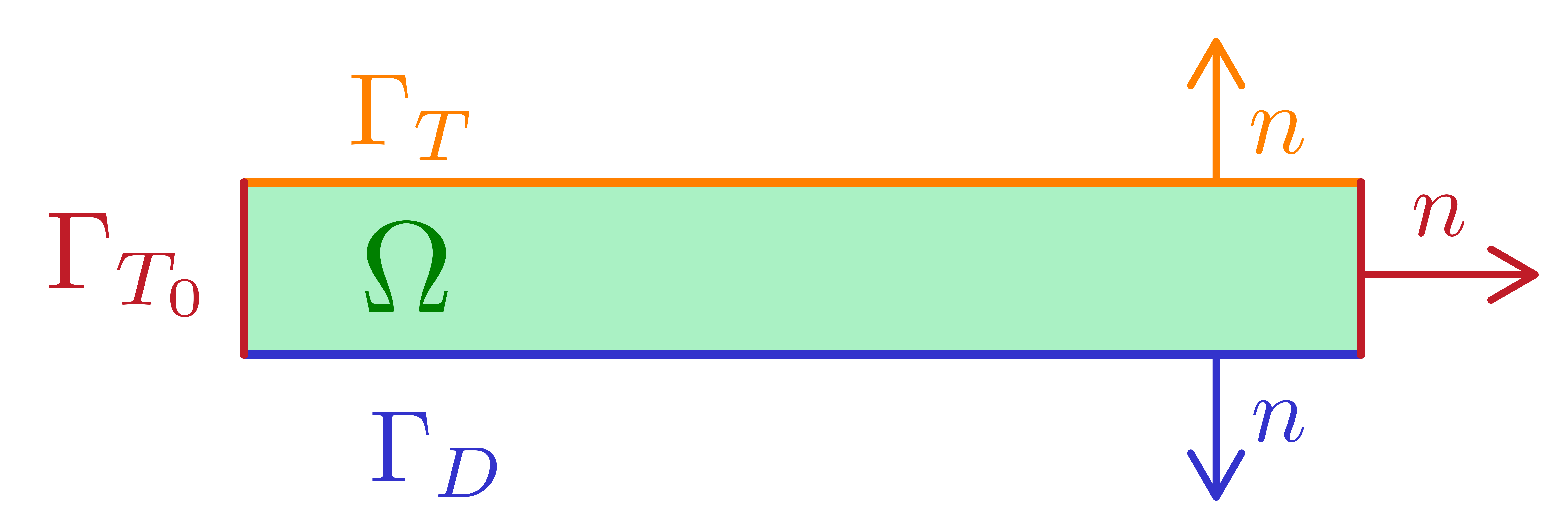}
    \caption{Substrate $\Omega$ with boundary $\partial \Omega = \Gamma_T \cup \Gamma_{T_0} \cup \Gamma_D$ and outer normal vectors $n$.}
    \label{fig:substrate}
\end{figure}

The balance of internal and body forces is expressed as 
$- \text{div } \sigma = f$ in $\Omega $
where the forces $f$ acting on the whole substrate, e.g., gravity, are neglected in the context of TFM. Therefore, we assume in the following $f=0$.

The traction stresses that the cell exerts onto the substrate can be modeled as surface forces acting on the boundary of the substrate by
$
    \sigma n = t  \text{ on } \Gamma_{T_*}
$
where $n$ is the outer normal vector, see Figure~\ref{fig:substrate}.
Since the cell applies forces only to the upper surface $\Gamma_T$, we have $t = 0$ on the side surface $\Gamma_{T_0}$. Therefore in the following, we will consider the traction stress $t$ only as a function on the upper surface $\Gamma_T$.
Furthermore, the displacement on the lower surface $\Gamma_D$ is zero because the substrate is fixed to the glass.
Altogether, given traction stresses $t \in L^2(\Gamma_T, \R^3)$ we obtain the mixed BVP of elasticity: 
\begin{align} 
\label{eq:mixedBVP}
    \begin{cases}
		- \text{div } \sigma(u) = 0 & \text{in } \Omega \\
  		\sigma(u) n = t & \text{on } \Gamma_T \\
        \sigma(u) n = 0 & \text{on } \Gamma_{T_0} \\    
		u = 0 & \text{on } \Gamma_D. 
    \end{cases}
\end{align}
Next we derive a weak formulation in the space 
\[
 H_{0,\Gamma_D}^1(\Omega, \R^3)\coloneqq \{ u \in H^1(\Omega, \R^3) : u|_{\Gamma_D} = 0 \}
\quad\mbox{with}\quad
\lVert v \rVert_{H^1_{0,\Gamma_D}(\Omega, \R^3)} \coloneqq      \lVert \nabla v \rVert_{L^2(\Omega, \R^3)}\,.
\]
It follows from Tartar's equivalence lemma~\cite{Tartar2007} that $\|\cdot\|_{H^1_{0,\Gamma_D}}$ is a Hilbert 
norm equivalent to the standard $H^1$-norm. 
The weak formulation of problem~\eqref{eq:mixedBVP} is given by
\begin{align}
    \label{eq:weakTFM}
            \text{find } &u \in H_{0,\Gamma_D}^1(\Omega, \R^3) \quad \mbox{s.t.}\quad 
            a(u,v) = l_t(v) \text{ for all } v \in H_{0,\Gamma_D}^1(\Omega, \R^3)
\end{align}
with the bilinear form $a$ and the linear form $l_t$ defined by
\begin{align*}
    a(u,v) \coloneqq \int_\Omega 2 \mu \varepsilon(u) : \varepsilon(v) + \lambda \diver(u) \diver(v) \dx, \quad
    l_t(v) \coloneqq \int_{\Gamma_T} t v \ds.
\end{align*}
Here $(M : M)\coloneqq \lvert M \rvert_{\mathrm{F}}^2 = \sum_{i,j=1}^n M_{ij}^2$ for matrices $M \in \R^{n\times n}$ is the Frobenius inner product with corresponding Frobenius norm $\lvert \cdot \rvert_{\mathrm{F}}$.

It follows from a standard application of the Lax-Milgram lemma and Korn's inequality that the parameter-to-state-map of linear TFM
\begin{align}
    \label{eq:tfmForward}
        A : L^2(\Gamma_T, \R^3) \to H^1_{0,\Gamma_D}(\Omega,\R^3), \quad
        t \mapsto u,
\end{align}
which is defined by \eqref{eq:weakTFM}, is  well-posed, linear, and bounded.

Then the inverse problem consists in recovering the traction stresses $t$ from a given displacement $u$.

\begin{remark}
\label{rem:forwardEmbeddingLin}
    The forward operator of linear TFM is given by the composition of the embedding to $L^2$ and parameter-to-state-map since we can only measure the $L^2$-error and norm of the measured displacement. All in all we get a forward operator
    \begin{align*}
        \hat{A} : L^2(\Gamma_T, \R^3) \to H^1_{0,\Gamma_D}(\Omega,\R^3) \hookrightarrow L^2(\Omega,\R^3), \quad
        t \mapsto u.
    \end{align*}
\end{remark}

We need to compute the adjoint operator to use regularization algorithms, e.g. in the normal equations, see Equation~\eqref{eq:normalEquations}.

\begin{lemma}
    The adjoint operator $\hat{A}^* : L^2(\Omega, \R^3) \to L^2(\Gamma_T, \R^3)$ is given by
    \begin{align*}
       \hat{A}^* w = \tr \phi
    \end{align*}
    where the function $\phi$ solves 
    \begin{align}
    \label{eq:adjointProblemWeak}
        a(\phi,v) = \langle w , v \rangle_{L^2(\Omega,\R^3)} \quad \forall v \in H^1_{0,\Gamma_D}(\Omega,\R^3)
    \end{align}
    and $\tr \phi$ denotes its trace on $\Gamma_T$.
\end{lemma}

\begin{proof}
Let a traction stress $t \in L^2(\Gamma_T, \R^3)$ be arbitrary and define $u \coloneqq \hat{A} t$. Then we obtain
    \begin{align*}
        \langle \hat{A} t , w \rangle_{L^2(\Omega,\R^3)} 
        &= \langle u, w \rangle_{L^2(\Omega,\R^3)} 
        = a(\phi,u)
        = a(u, \phi) \\
        &= \langle t, \tr \phi \rangle_{L^2(\Gamma_T, \R^3)} 
        = \langle t, \hat{A}^* w \rangle_{L^2(\Gamma_T, \R^3)}.
    \end{align*}
It remains to show that~\eqref{eq:adjointProblemWeak} has a unique solution, which can be done again by Lax-Milgram since the bilinear form $a$ is continuous and coercive and the linear form $l^*_w(v) \coloneqq  \langle w , v \rangle_{L^2(\Omega,\R^3)}$ is continuous.
\end{proof}


\section{Nonlinear pure 2D traction force microscopy}
\label{ch2:nonlinearTFM}

The linear model works well for small strains and linear materials. If we have larger strains we cannot assume anymore that the nonlinear part $(\nabla u)^{\top} \nabla u$ in~\eqref{eq:Green-Lagrange} is small and neglect it. Therefore, we use the Green-Lagrange strain tensor~\eqref{eq1:nonlinearstrain} instead of the linearized strain tensor. 
Furthermore, if a nonlinear material is used, we need to find a suitable constitutive equation as a material law that describes the stress $\sigma(u)$. 

These changes lead to a nonlinear PDE and thus to a nonlinear inverse problem. 

\subsection{Mathematical 2D model and inverse problem for a nonlinear material}
\label{sec2:modelNonlin}

In the nonlinear case, we consider a pure two-dimensional model, i.e., we only describe the surface of the substrate to which the cell adheres (see Figure \ref{fig:3Dto2D}). 
For a deformation $\varphi : \R^2 \to \R^2$, we define, as before in Section~\ref{sec:elasticity}, the displacement vector $u(x) = \varphi(x) - x$, the Green-Lagrange strain tensor $E = \frac{1}{2} \left( \nabla u + (\nabla u)^\top + (\nabla u)^\top \nabla u \right)$ and the stress tensor $\sigma (u)$ given by a suitable material law to be specified later.

If the undeformed region can be modeled by a bounded, open, and connected domain $\Omega \subset \R^2$ with a Lipschitz continuous boundary $\partial \Omega$, the problem of traction force microscopy can be described by the boundary value problem
\begin{align}
    \label{eq2:BVP}
    \left\lbrace\,
    \begin{array}{@{}l@{}l@{\quad}l@{}}%
            - \diver (\sigma(u)) &= T \coloneqq \frac{t}{h} &\text{in } \Omega \\
            u &= 0 &\text{on } \partial \Omega
    \end{array} \right.
\end{align}
for traction stresses $t$ with zero boundary conditions,
effective thickness $h$, and the force density $T$. The effective thickness $h$ is the thickness beyond which the horizontal displacements disappear~\cite{Michel2013}. The first equation in~\eqref{eq2:BVP} is called the thickness-averaged condition for stress equilibrium~\cite{Dembo1996, Michel2013, landau1986}.

In the linear case, we considered the mixed boundary value problem of elasticity with traction and displacement boundary conditions in 3D. If we change the model to the two-dimensional case, we neglect that the substrate is three-dimensional and only consider the upper surface of the substrate as our domain with the thickness-averaged condition for stress equilibrium~\cite{Dembo1996}. This model has been used in, e.g.,~\cite{Dembo1996, Michel2013}. However the standard model in physics assumes an infinitely large halfspace leading to a model that is not purely nonlinear, see e.g.~\cite{Dembo1999, Schwarz2015}.

By consequence, in the 2D case here, the old upper surface $\Gamma_T$ now corresponds to the entire new domain $\Omega$, see Figure~\ref{fig:3Dto2D}.

\begin{figure}[h]
    \centering
    \includegraphics[width=10cm]{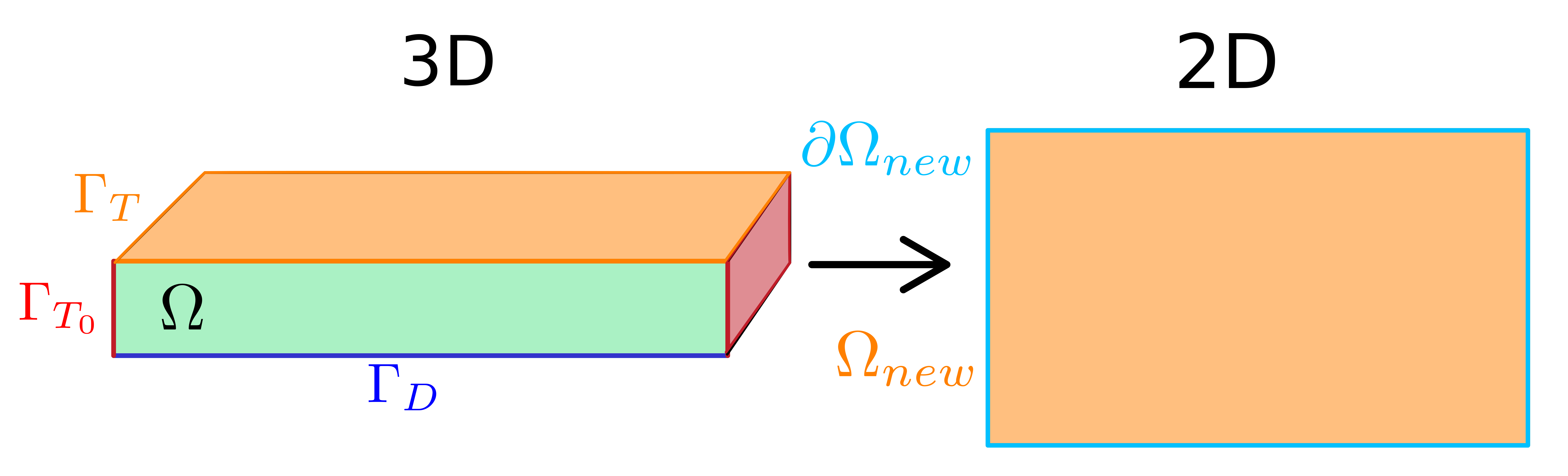}
    \caption{Change of TFM setting from 3D to 2D substrate.}
    \label{fig:3Dto2D}
\end{figure}

Due to this change in the mathematical setup, the traction stresses $t$ at the surface of the three-dimensional substrate in the 2.5D case as in~\eqref{eq:mixedBVP} become part of the volume force density in the 2D case as in~\eqref{eq2:BVP}, see \cite{Dembo1996}. This leads to a displacement boundary value problem in 2D which is easier to solve in the nonlinear case than the mixed boundary value problem. To get a well-defined problem we need to add boundary conditions. Therefore we assume that the substrate is larger than the support of the traction stresses, such that displacement at the boundary is negligible and we may assume zero Dirichlet boundary conditions. This assumption reflects the experimental situation well and holds, e.g., for sparsely seeded fibroblast cells that do not form connections to neighboring cells.

We can formulate the parameter-to-state-map of nonlinear TFM similarly as in the linear case as the map
\begin{align}
    \label{eq2:tfmForward}
        S : X \to Y, \quad
        T \mapsto u
\end{align}
where the displacement $u$ is the solution to the BVP~\eqref{eq2:BVP}. The Banach spaces $X$ and $Y$ are determined later in Remark~\ref{rem:funcSpace} using the existence and uniqueness results.  In contrast to the 2.5D case, where the operator $A$ in~\eqref{eq:tfmForward} maps a traction stress $t$ to a displacement $u$, in the 2D case the operator $S$ maps a force density $T$ to a displacement $u$. However, it is easy to convert traction stress $t$ to force density $T$ and vice versa by
\begin{align}
    \label{eq2:TractionStressToForceDensity}
    T = \frac{t}{h}.
\end{align}

\subsection{Well-posedness of the forward operator}
\label{sec2:well-posedness}
 
To get a well-defined forward operator $S$, the BVP~\eqref{eq2:BVP} needs to be solvable and the solution has to be locally unique. Unlike in the linear case in Section~\ref{ch:linTFM}, where a global unique existence result can be found easily in the literature, the problem of unique existence is more complicated in the nonlinear case.

In the following we will first show the existence of a solution via the minimization of energy in~\ref{s2:existence}, then state a local uniqueness result via the Implicit Function Theorem in~\ref{s2:uniqueness} under additional regularity conditions and finally argue for a suitable choice of the Banach spaces $X$ and $Y$ in~\eqref{eq2:tfmForward}.

\subsubsection{Existence of a solution via minimization of the energy}
\label{s2:existence}

In 1976 John Ball proved the existence of a solution for general boundary value problems of elasticity in 1D, 2D, and 3D for nonlinear elasticity \cite{Ball1976ConvexityCA}. To use his results directly many conditions on the used functions must be fulfilled. He also gives a proof for the mixed BVP in 3D. Similar proofs can be found, e.g., in \cite[chapter 7]{CiarletI} or \cite[section 62.13]{Zeidler1997NonlinearFA4}. However we are interested in the displacement BVP in 2D. Having only Dirichlet boundary conditions is a subcase of the mixed BVP, but to get the result in 2D suitable adaptions to the proof are necessary. The main difference lies in the definition of polyconvexity and coercivity of the stored energy function, see equation~\eqref{eq:hyperelastic}, depending on the dimension $n$ of the deformation gradient $F \in \R^{n \times n}$. 

\begin{definition}[Polyconvexity]
\label{def:polyconvex}
    A stored energy function $W$ is called polyconvex if there exists a convex function $P$ such that
    \begin{alignat*}{2}
            W(F) &= P(F, \det F) &&\text{ if } n=2, \\
            W(F) &= P(F, \adj F, \det F) &&\text{ if } n=3.
    \end{alignat*}
    For $n=3$ we need the adjugate matrix defined by $\adj F = \det F \cdot F^{-1}$.
\end{definition}

\begin{definition}[Coercivity]
\label{def:coercivity}
    A polyconvex stored energy function $W$ in $\mathbb{R}^{2\times 2}$ is called coercive if
    \begin{align*}
        W(F) =  P(F,\det F) \geq C(\lvert F \rvert_{\mathrm{F}}^p + \lvert \det F \rvert^s) + D  
    \end{align*}
    holds for the convex function $P$ from Definition~\ref{def:polyconvex} for $p \geq 2$, $s > 1$, $r \geq \frac{p}{p-1}$, $C >0$ and $D \in \R$.
\end{definition}

In the 3D case the deformation gradient $F$ and its determinant and adjugate, $\det F$ and $\adj F$, describe the deformation of line, surface, and volume elements. Furthermore, in the 2D case, it holds $\lvert \adj F \rvert_{\mathrm{F}} = \lvert F \rvert_{\mathrm{F}}$. Therefore in the 2D case the adjugate, $\adj F$, is not needed. 

The next theorem ensures the existence of a solution for general polyconvex and coercive stored energy functions.

\begin{theorem}{(Existence in 2D for the displacement BVP)}
\label{th2:existence Ball}

\noindent
Let $W$ be a polyconvex and coercive stored energy function such that
\begin{align}
    \label{eq:existence,limit}
    \lim_{\det F \to 0} P(F,\det F) = \infty
\end{align}
for the convex function $P : \M^2 \times (0, \infty) \to \R$ from Definition~\ref{def:polyconvex} with $p$ and $s$ as in Definition~\ref{def:coercivity} and the space of matrices $\M^2 \coloneqq \{ M \in \R^{2 \times 2} \}$.

Let the force density $T \in L^p(\Omega,\R^2)$ and a finite reference deformation $u_0 \in V$ be given with
    \begin{align}
    \label{eq2:V}
        V \coloneqq \{ u \in W^{1,p}(\Omega,\R^2)\,| \, \det (I + \nabla u) \in L^s(\Omega,\R^2), \det (I + \nabla u) > 0 \text{ a.e. in } \Omega\}.
    \end{align}
Then for 
\begin{align}
\label{eq2:energy}
    G(u) = \int_{\Omega} W(F) \dx -\int_{\Omega} T u \dx= \int_{\Omega} W(I + \nabla u(x)) \dx - \int_{\Omega} T u \dx
\end{align}
the problem 
\begin{align}
    \label{eq2:MinProb}
    &G(u) = \min!, \quad u \in U \coloneqq \bigl\{ u \in V \bigr| \hspace{1ex} u|_{\partial \Omega} = u_0|_{\partial \Omega} \bigr\}
\end{align}
has at least one solution if $G(u_0) < \infty$.

\end{theorem}

\begin{proof}
    We use the direct method to prove the theorem and follow the proof of the 3D case from \cite[chapter 62.13]{Zeidler1997NonlinearFA4} and \cite[Theorem 7.7-1]{CiarletI} with suitable adaptions for 2D, mainly regarding coercivity and polyconvexity.
    
    The main idea is to define a minimizing sequence of the energy $G$ in~\eqref{eq2:energy}. This sequence is then bounded (step ii) by the coercivity of the stored energy function and thus it converges to an element $u^*$ (steps iii). Due to the polyconvexity, the energy is sequentially weakly lower semi-continuous which yields $G(u^*) \leq \inf_{u \in U} G(u)$ (step iv). The element $u^*$ also belongs to the set $U$ due to~\eqref{eq:existence,limit} (step v). The detailed proof can be found in Appendix~\ref{Proof_Existence}.
\end{proof}

\subsubsection{Local uniqueness of the solution via the Implicit Function Theorem}
\label{s2:uniqueness}
    The proof of uniqueness via the Implicit Function Theorem is based on the Banach algebra property of the spaces $W^{m,p}(\Omega, \R^2)$, i.e., the fact that the product of two functions of the space lies again in the space with a corresponding norm bound, which only holds true for $mp > n$, where $n=2$ is the space dimension, see, e.g., Theorem 6.1-4 from \cite{CiarletI}.
    
    In the proof a known admissible solution in the space $W^{m,p}(\Omega,\R^2)$ is locally uniquely continued. We call a deformation state $u$ admissible if the linearization of the original equation in $u$ yields a linear strongly elliptic system with a unique solution, see \cite[chapter 61.12]{Zeidler1997NonlinearFA4}. Then in a small neighborhood of an admissible solution and the associated traction stress the boundary value problem of nonlinear elasticity~\eqref{eq2:BVP} has a unique solution. 

     A key ingredient to prove this statement is the Sobolev embedding $W^{1,p}(\Omega,\R^2) \hookrightarrow C^0(\Omega \cup \partial \Omega,\R^2)$ for $p > 2$, see e.g.~\cite[Theorem 6.1-3]{CiarletI}. Thus, in analogy to the $3D$-case in~\cite[Theorem 6.7-1]{CiarletI},~\cite[Theorem 61.F]{Zeidler1997NonlinearFA4} or~\cite[Chapter IV]{Valent1988BoundaryVP}, we obtain the following result.
     \begin{theorem}{(Existence and uniqueness)}
         \label{thm:existence_uniqueness}
        Let $\Omega \subset \R^2$ be a domain with boundary $\partial \Omega$ of class $C^2$ and let the stored energy function be $C^\infty$. If we know an admissible solution $\bar{u} \in W^{2,p}(\Omega) \cap W^{1,p}_0(\Omega)$ for $p>2$ with respective force density $\bar{T}$, then there exist neighborhoods
        \begin{align*}
            V(\bar{u}) \in W^{2,p}(\Omega) \cap W^{1,p}_0(\Omega), \qquad W(\bar{T}) \in L^p(\Omega)
        \end{align*}
        such that for each $T \in W$, the BVP $-\diver(\sigma(u)) = T$ has exactly one solution $u \in V$. 
        Moreover, the linearized operator $W^{2,p}(\Omega)\to L^p(\Omega)$, $h\mapsto -\diver(\sigma'(u)h)$ is bounded and 
        boundedly invertible. 
     \end{theorem}


    It can be shown that no deformation state $u \in W^{1,p}(\Omega,\R^2)$ can be admissible \cite[chapter 4]{Valent1988BoundaryVP}. Since the existence result from Theorem~\ref{th2:existence Ball} just gives a solution in the space $W^{1,p}(\Omega,\R^2)$, the regularity of this solution has to be improved to be continued uniquely as it is required in the uniqueness proof. This might be done in a similar manner as described in~\cite{Morrey1966}.



\subsection{Finding a suitable material law}
\label{sec2:matLaw}

Next, we need to find a suitable material law describing the stress $\sigma(u)$. As we saw in the previous Section~\ref{sec2:well-posedness}, it is important for the existence of a solution that the stored energy function $W(F)$ of the material law be polyconvex. 
Unfortunately the Neo-Hookean law typically used in physics~\cite{Schwarz2015,ToyjanovaNonlinearNeo-Hooke,Sanz-HerreraNonlinearNeo-Hooke} for nonlinear TFM is not polyconvex~\cite{SuchokiNeoHookeanNotPolyconvex} and thus we cannot guarantee that a solution to the respective BVP exists. For this reason, we pick a different material law from the family of polyconvex Ogden materials.

To ensure local existence and uniqueness, the stored energy function of a homogeneous, isotropic, hyper\-elastic~\eqref{eq:hyperelastic} material should agree with the expansion
\begin{align}
\label{eq2:expansionNaturalState}
    W(F) = \frac{\lambda}{2} (\tr E)^2 + \mu \tr (E^2) + o(\lVert E \rVert^2), \quad E = \frac{1}{2} (F^{\top} F - I),
\end{align}
with the Lam\'e parameters $\lambda$ and $\mu$, see \eqref{eq:lame}, if it is rewritten in terms of the Green-Lagrange tensor $E$ near a natural state, i.e., an unstressed state in which all body forces vanish. 
When deriving rules for the equations of elasticity, it can be shown that a constitutive equation of an isotropic, homogeneous material whose reference configuration is a natural state has to fulfill the expansion $\sigma(u) = F \left(\lambda (\tr E(u)) I + 2 \mu E(u) + o(E) \right)$ near a natural state with the Green-Lagrange tensor $E(u)$ as in~\eqref{eq:Green-Lagrange}. By using the relation~\eqref{eq:hyperelastic} between the constitutive equation and the stored energy function and postulating further differentiability assumptions, the expansion~\eqref{eq2:expansionNaturalState} follows.
Further details and a physical intuition can be found in Chapter 4.5 of~\cite{CiarletI}. 

Adapting Theorem 4.10-2 from~\cite{CiarletI} to the two-dimensional case we get the following result.

\begin{theorem}[Constitutive Equation]
\label{th2:constEq}
    Let $\lambda, \mu > 0$ be the given Lam\'e constants. Then the stored energy function
    \begin{align}
        \label{eq2:storedEnergy}
        W(F) \coloneqq \frac{\mu}{2} \lvert F \rvert_{\mathrm{F}}^2 + \frac{\lambda}{4} (\det F)^2 - \Big(\mu + \frac{\lambda}{2}\Big) \ln (\det F) - \frac{3\mu}{2} - \frac{\lambda}{4}
    \end{align}
    is polyconvex and satisfies the expansion
    \begin{align}
    \label{eq2:expasion2}
        W(F) = \frac{\lambda}{2} (\tr E)^2 + \mu \tr E^2 + \mathcal{O}(\lvert E \rvert_{\mathrm{F}}^3).
    \end{align}
    Furthermore, it satisfies the coercivity inequality
    \begin{align}
    \label{eq2:Wcoercive}
        W(F) \geq C (\lvert F \rvert_{\mathrm{F}}^2 + (\det F)^2) + D
    \end{align}
    for constants $C > 0$, $D \in \R$, if the Lam\'e constants fulfill the condition 
    \begin{align}
    \label{eq2:constEqLameCondition}
        \lambda > \frac{2 \mu}{\euler - 1}.
    \end{align}
    with Euler's number $\euler$.
\end{theorem}

\begin{proof}
    We write the stored energy function~\eqref{eq2:storedEnergy} in the general form
    \begin{align}
    \label{eq2:generalStoredEnergy}
        W(F) = a \lvert F \rvert_{\mathrm{F}}^2 + \Gamma(\det F) + b
    \end{align}
    with a function $\Gamma$ of the form $\Gamma(\delta) = c \delta^2 - d \ln \delta$ and constants $a, c, d > 0$,  $b \in \R$. Then we determine the constants $a,b,c,d$ such that the stored energy function $W$ fulfills the conditions~\eqref{eq2:expasion2} and~\eqref{eq2:Wcoercive} as well as polyconvexity:

 \textbf{(i)}   We prove the following equations for two-dimensional matrices by simple calculations using the big $\mathcal{O}$-notation:
    \begin{align*}
        I + 2 E &= F^{\top} F \\
        \det F^{\top} F &=  \det ( I + 2 E) = 1 + 2 (\tr E)^2 - 2 \tr(E^2) + 2 \tr E + \mathcal{O}(\lvert E \rvert_{\mathrm{F}}^3) \\
        \Gamma(\det F) &= \Gamma( (\det F^{\top} F)^{\frac{1}{2}}) 
        \stackrel{(*)}{=} \Gamma \bigl(1 + \tr E + \frac{1}{2} (\tr E)^2 - \tr E^2 + \mathcal{O}(\lvert E \rvert_{\mathrm{F}}^3) \bigr) \\
        &\stackrel{(**)}{=} \Gamma(1) + \Gamma'(1) [\tr E + \frac{1}{2} (\tr E)^2 - \tr E^2] + \frac{1}{2} \Gamma''(1) (\tr E)^2 + \mathcal{O}(\lvert E \rvert_{\mathrm{F}}^3).
    \end{align*}
    Equality (*) holds due to the equality 
    \begin{align*}
    	\left( 1 + \tr E + \frac{1}{2} (\tr E)^2 - \tr E^2 + \mathcal{O}(\lvert E \rvert_{\mathrm{F}}^3)\right)^2 
    	= (\det F^\top F)
    \end{align*} and we get equality (**) by applying the Taylor expansion to the function $\Gamma$.
    Now we simply compare the expansion~\eqref{eq2:expasion2} with the general form of the stored energy function~\eqref{eq2:generalStoredEnergy}. By using the derived equations for $\lvert F \rvert_{\mathrm{F}}^2$ and $\Gamma(\det F)$ we arrive at the system of equations
    \begin{align*}
        3 a + \Gamma(1) + b &= 0, \\
        2 a + \Gamma'(1) &= 0, \\
        \Gamma'(1) + \Gamma''(1) &= \lambda, \\
        - \Gamma'(1) &= \mu.
    \end{align*}
    The expressions $\Gamma'(1)$ and $\Gamma''(1)$ are uniquely determined by these equations. Finally combining these equations with $\Gamma(1) = c$, $\Gamma'(1) = 2c -d$, and $\Gamma''(1) = 2c + d$ yields 
   \begin{align*}
   a = \frac{\mu}{2}, \quad
   c = \frac{\lambda}{4}, \quad
   d = \mu + \frac{\lambda}{2} , \quad
   b = - \frac{3 \mu}{2} - \frac{\lambda}{4}
   \end{align*}
and we arrive at~\eqref{eq2:generalStoredEnergy}.

 \textbf{(ii)} From the definition of the deformation gradient $F = \nabla \varphi + I$ with the deformation $\varphi$, we get the condition $\det F > 0$ due to mass conservation and \eqref{eq2:generalStoredEnergy} is well-defined. 

 For coercivity we need an estimate of the form 
$cx^2 - d \ln(x) \geq B x^2$, or equivalently 
 \begin{align}
 \label{eq:proofConstEq(1)}
     x^2 \geq \frac{d}{c-B} \ln(x) 
 \end{align}
 for a positive constant $B > 0$ and $x = \det F > 0$.
 According to $x^2 \geq 2\euler \ln(x)$ with Euler's number $\euler$, 
 the constant $B$ has to fulfill 
$
     0 < B \leq c - \frac{d}{2\euler}.
$
 This is possible, i.e., $c - \frac{d}{2\euler} > 0$, if the Lam\'e constants fulfill the relation
 \begin{align*}
     \lambda > \frac{2 \mu}{\euler - 1} > 1.16 \mu.
 \end{align*}
Using equation~\eqref{eq:proofConstEq(1)} we get coercivity by
    \begin{align*}
        a \lvert F \rvert_{\mathrm{F}}^2 + c (\det F)^2 - d \ln(\det F) + b 
        \geq &a \lvert F \rvert_{\mathrm{F}}^2 + B (\det F)^2 + b \\
        \geq &\min(a, B) \left(\lvert F \rvert_{\mathrm{F}}^2 + (\det F)^2\right) + b .
    \end{align*}
    Since the squared norm (see, e.g., \cite{HARTMANN20032767}) and the function $x \mapsto c x^2 - d\ln(x)$ (second derivative is positive) are convex on the interval $(0,\infty)$ and the sum of convex functions is again convex, the stored energy \eqref{eq2:generalStoredEnergy} is polyconvex.
\end{proof}

In TFM the condition~\eqref{eq2:constEqLameCondition} for the Lam\'e parameters is not a problem. For a standard choice with a Poisson ratio $\nu = 0.45$ and a Young's modulus $\tilde{E}= 10 000 Pa$, see~\cite{Schwarz2015}, we get Lam\'e constants $\lambda \approx 31 034 Pa$ and $\mu = 3448 Pa$ which clearly fulfill condition~\eqref{eq2:constEqLameCondition}. Now, that we have fixed a stored energy function $W$ determining the material law, we need to compute the constitutive equation determining the stress $\sigma(u)$. The stored energy function $W$ and the stress $\sigma$ are related via $\sigma = \frac{\partial W}{\partial F}$, see~\eqref{eq:hyperelastic}.

Using the matrix derivation rule for the determinant 
$
\frac{\partial \det X}{\partial X} = \det(X) X^{-\top},
$
see, e.g., \cite{MatrixDeriv}, and the chain rule on the stored energy function $W(F)$ in~\eqref{eq2:storedEnergy} we then get
\begin{align}
    \label{eq2:constitutiveEquation}
    \sigma(u) = \mu F + \frac{\lambda}{2} (\det F)^2 F^{-\top} - (\mu + \frac{\lambda}{2}) F^{-\top}
\end{align}
where we use the notation $F^{-\top} = \left( F^{-1} \right)^{\top}$.

\begin{remark}[Function space setting]
\label{rem:funcSpace}

\noindent
The stored energy function from Theorem~\ref{th2:constEq} for $p=s>2$ and the boundary conditions $u_0 = 0$ fulfills the conditions of Theorem~\ref{th2:existence Ball} for the existence of a solution and of Theorem~\ref{thm:existence_uniqueness} for the uniqueness of this solution. This means that the minimization problem~\eqref{eq2:MinProb} has at least one solution $u \in W^{2,p}(\Omega,\R^2) \cap W^{1,p}_0(\Omega,\R^2) $ for a force density $T \in L^p(\Omega,\R^2)$. Then we can choose $X = L^p(\Omega,\R^2)$ and $Y = W^{2,p}(\Omega,\R^2) \cap W^{1,p}_0(\Omega,\R^2)$ in the formulation of the parameter-to-state-map~\eqref{eq2:tfmForward}.

We may incorporate further prior information into the 
space $X$: E.g., we may include 
vanishing forces on the boundary $\partial\Omega$, 
and higher regularity by choosing 
$X =  H_0^1(\Omega,\R^2)$, equipped with the 
norm $\|v\|_X:=\|\nabla v\|_{L^2}$ (see \cite[Chapter 13.2]{leoni2009} for the norm property). By choosing the space $X = H_0^1(\Omega,\R^2) \hookrightarrow L^p(\Omega,\R^2) $ we can also remain in a Hilbert space setting.

As in the linear case in Remark~\ref{rem:forwardEmbeddingLin} the forward operator is the composition of the parameter-to-state-map and the embedding to $L^2(\Omega,\R^2)$. Then we get the forward operator of nonlinear 2D TFM by
\begin{align*}
    \hat{S} :  H_0^1(\Omega,\R^2) \hookrightarrow L^p(\Omega,\R^2) \to W^{2,p}(\Omega, \R^2) \cap W^{1,p}_0(\Omega,\R^2) \hookrightarrow L^2(\Omega, \R^2), \quad T \mapsto u.
\end{align*}
 The embeddings hold due to the Rellich-Kondrachov embedding theorem, see~\eqref{eq2:proofExistenceEmbedding} in the Appendix.



\end{remark}

\subsection{The Fr{\'e}chet derivative and its adjoint}
\label{sec2:Frechet}

Since the operator $S$ is nonlinear, we need to compute the Fr\'echet derivative of $S$, leading to a linearized problem, and its adjoint to solve the inverse problem with regularization algorithms as described in Section~\ref{subsec:numericsNonlinear2D}.

\begin{theorem}[Fr\'echet derivative]
    The Fr\'echet derivative of the operator $S$ that maps a given force density $T$ to the displacement $u$ in the function space setting from Remark~\ref{rem:funcSpace} is well-defined and given by $S'(T)h = v$, where the function $v$ solves
    \begin{align}
        \label{eq2:Frechet}
        \left\lbrace\,
        \begin{array}{@{}r@{}l@{\quad}l@{}}%
            - \mu \diver (\nabla v) - (\mu + \frac{\lambda}{2}) \diver \Bigl(F^{-\top} (\nabla v)^{\top} F^{-\top}\Bigr)  & & \\
            -  \frac{\lambda}{2} \diver \Bigl( 2 \det(F)^2 \tr(F^{-1} \nabla v) F^{-\top} - \det(F)^2 F^{-\top} (\nabla v)^{\top} F^{-\top} \Bigr) &= h & \text{ in } \Omega \\
            v &= 0 & \text{ on } \partial \Omega
            \end{array} \right.
    \end{align}
    with the deformation gradient $F = I + \nabla u$ and $u = S(T)$ solves the original problem~\eqref{eq2:BVP}.
\end{theorem}

\begin{proof}
    With Theorem~\ref{thm:existence_uniqueness} the Fr\'echet derivative exists and can be calculated using the Implicit Function Theorem. For this, we define the mapping 
    $e:W^{2,p}(\Omega)\cap W^{1,p}_0(\Omega) \times L^p(\Omega)
    \to L^p(\Omega)$ by
    \begin{align*}
        e(u,T) = 
        - \diver (\sigma(u)) - T.  
    \end{align*} 
    Then the condition $e(S(T),T) = 0$ and the Implicit Function Theorem give us that the derivative of the operator $S$ is locally given by 
    \begin{align}
    \label{eq2:implicitFuncInverse}
        S'(T)h = -\frac{\partial e}{\partial u}(S(T),T)^{-1} \circ \frac{\partial e}{\partial T}(S(T),T)h .
    \end{align} 
    This is equivalent to the linearized problem
    \begin{align*}
        \frac{\partial e}{\partial u}(S(T),T) [S'(T)h] = - \frac{\partial e}{\partial T}(S(T),T) h
    \end{align*}
    Now we can compute the partial derivatives of the mapping $e$
    \begin{align*}
        \frac{\partial e}{\partial T}(S(T),T) h = 
        - h
, \quad
        \frac{\partial e}{\partial u}(S(T),T) v = 
        \varphi'(s)|_{s = 0} 
    \end{align*}
    with 
    \begin{align*}
        \varphi(s) = &- \diver(\sigma(u + sv)) - T  = (\mu + \frac{\lambda}{2}) \diver \Bigl((I + \nabla u + s\nabla v)^{-\top} \Bigr)\\
        &-\mu \diver(\nabla u + s\nabla v) - \frac{\lambda}{2}  \diver \Bigl(\det(I + \nabla u + s \nabla v)^2 (I + \nabla u + s \nabla v)^{-\top} \Bigr)
    \end{align*}
    for $u = S(T)$.
    By using the chain rule and matrix differentiation rules, see e.g. \cite{MatrixDeriv}, we arrive at
    \begin{align*}
        \varphi'(0) = &-\frac{\lambda}{2} \diver \Bigl( 2 \det(F)^2 \tr(F^{-1} \nabla v) F^{-\top} - \det(F)^2 F^{-\top} (\nabla v)^{\top} F^{-\top} \Bigr) \\
            & - \mu \diver (\nabla v) - (\mu + \frac{\lambda}{2}) \diver \Bigl(F^{-\top} (\nabla v)^{\top} F^{-\top}\Bigr). \qedhere
    \end{align*}   
\end{proof}

\begin{theorem}
    The Fr\'echet derivative $S'(T)$ is self-adjoint in $L^2(\Omega)$ for all $T$.
\end{theorem}
\label{thm:frechet}

\begin{proof} 
    To compute the adjoint operator we use the expression for the Fr\'echet derivative that we get from the Implicit Function Theorem~\eqref{eq2:implicitFuncInverse}. Then its adjoint is given by
        $$\bigl( S'(T) \bigr)^* g 
        = - \left( \frac{\partial e}{\partial T}(S(T),T) \right)^*
       \left( \left( \frac{\partial e}{\partial u}(S(T),T) \right)^{-1} \right)^*
        g$$
    and can be calculated via  
    \begin{align}
        \label{eq2:adjointFormel}
        \bigl( S'(T) \bigr)^* g 
        = - \left( \frac{\partial e}{\partial T}(S(T),T) \right)^* w
    \end{align}
        where the function $w$ solves
    $
        \left( \frac{\partial e}{\partial u}(S(T),T) \right)^{*} 
        w = g
    $.
    Now we compute the adjoint operator $\left( \frac{\partial e}{\partial u}(S(T),T) \right)^{*} $ by applying partial integration, using the boundary conditions $v,w=0$ on $\partial \Omega$ and the equality $A : B = \tr(A^{\top} B)$. We have
    \begin{align*}
        \left\langle \frac{\partial e}{\partial u}(S(T),T)   v, w \right\rangle_{L^2(\Omega)} 
        = &\int_\Omega -\mu \diver(\nabla v) w  
         - \left(\mu + \frac{\lambda}{2} \right) \diver \biggl(F^{-\top} (\nabla v)^{\top} F^{-\top} \biggr) w\\
        &- \frac{\lambda}{2} \diver \biggl( 2 (\det F)^2 \tr(F^{-1} \nabla v) F^{-\top} - (\det F)^2 F^{-\top} (\nabla v)^{\top} F^{-\top} \biggr) w \dx \\
        = &\int_\Omega \Big(\mu \nabla v : \nabla w  
         + \left(\mu + \frac{\lambda}{2} - \frac{\lambda}{2} (\det F)^2 \right) \biggl(F^{-\top} (\nabla v)^{\top} F^{-\top} \biggr) : \nabla w\\
        &+ \lambda  (\det F)^2 (F^{-\top}: \nabla v) (F^{-\top} : \nabla w)\Big) \, \dx.
    \end{align*}
    Since the trace is invariant under circular shifts, we get 
    \begin{align*}
        (F^{-\top} (\nabla v)^{\top} F^{-\top}) : \nabla w
        &= \tr \biggl( F^{-1} (\nabla v) F^{-1} (\nabla w) \biggr) \\
        &= \tr \biggl( (\nabla v) F^{-1} (\nabla w) F^{-1} \biggr)
        = \nabla v : (F^{-\top} (\nabla w)^{\top} F^{-\top}), 
    \end{align*}
    and by a partial integration analogous to the one above we see that
    \begin{align*}
        \left\langle  \frac{\partial e}{\partial u}(S(T),T)   v, w \right\rangle_{L^2(\Omega)} 
        =
        \left\langle   v,  \frac{\partial e}{\partial u}(S(T),T)  w \right\rangle_{L^2(\Omega)},
    \end{align*}
    i.e. the operator $\frac{\partial e}{\partial u}(S(T),T)$ is self-adjoint.

     Since the operator $\frac{\partial e}{\partial T}(S(T),T)$ only changes the sign of its argument, i.e. $\frac{\partial e}{\partial T}(S(T),T) v = - v$, it is self-adjoint and, with equation~\eqref{eq2:adjointFormel}, we obtain 
     \begin{align*}
         S'(T)^* = - \frac{\partial e}{\partial T}(S(T),T) \left(\frac{\partial e}{\partial u}(S(T),T) \right)^{-1} 
         = - \left(\frac{\partial e}{\partial u}(S(T),T) \right)^{-1} \frac{\partial e}{\partial T}(S(T),T) = S'(T),
     \end{align*}
     i.e., the Fr\'echet derivative $S'(T)$ is self-adjoint.
\end{proof}

\begin{remark}
As pointed out in Remark \ref{rem:funcSpace}, we will consider a forward operator $\hat{S}$ defined 
on $H^1_0(\Omega,\mathbb{R}^2)$, equipped with the norm $\|T\|_{H^1_0}:=\|\nabla T\|_{L^2}$. Then 
$\hat{S}'(T) =  S'(T)\circ j$ with the operator $S'(T)$ defined on $L^2(\Omega,\R^2)$ from Theorem~\ref{thm:frechet} and the embedding $j:H^1_0(\Omega,\mathbb{R}^2\hookrightarrow L^2(\Omega,\mathbb{R}^2)$. Then we have $\tilde{S}'[T]^*= j^*\circ S'(T)^* = j^*\circ S'(T)$. A straightforward computation 
shows that $j^*:L^2(\Omega,\mathbb{R}^2)\to H^1_0(\Omega,\mathbb{R}^2)$ is given by $j^*=(-\Delta)^{-1}$ 
with the vector Dirichlet Laplacian $\Delta$. 
\end{remark}


\section{Numerical results}

\label{sec:numerics}

\begin{figure}[H]
\centering
\begin{subfigure}{0.33\textwidth}
\centering
\includegraphics[width=\linewidth]{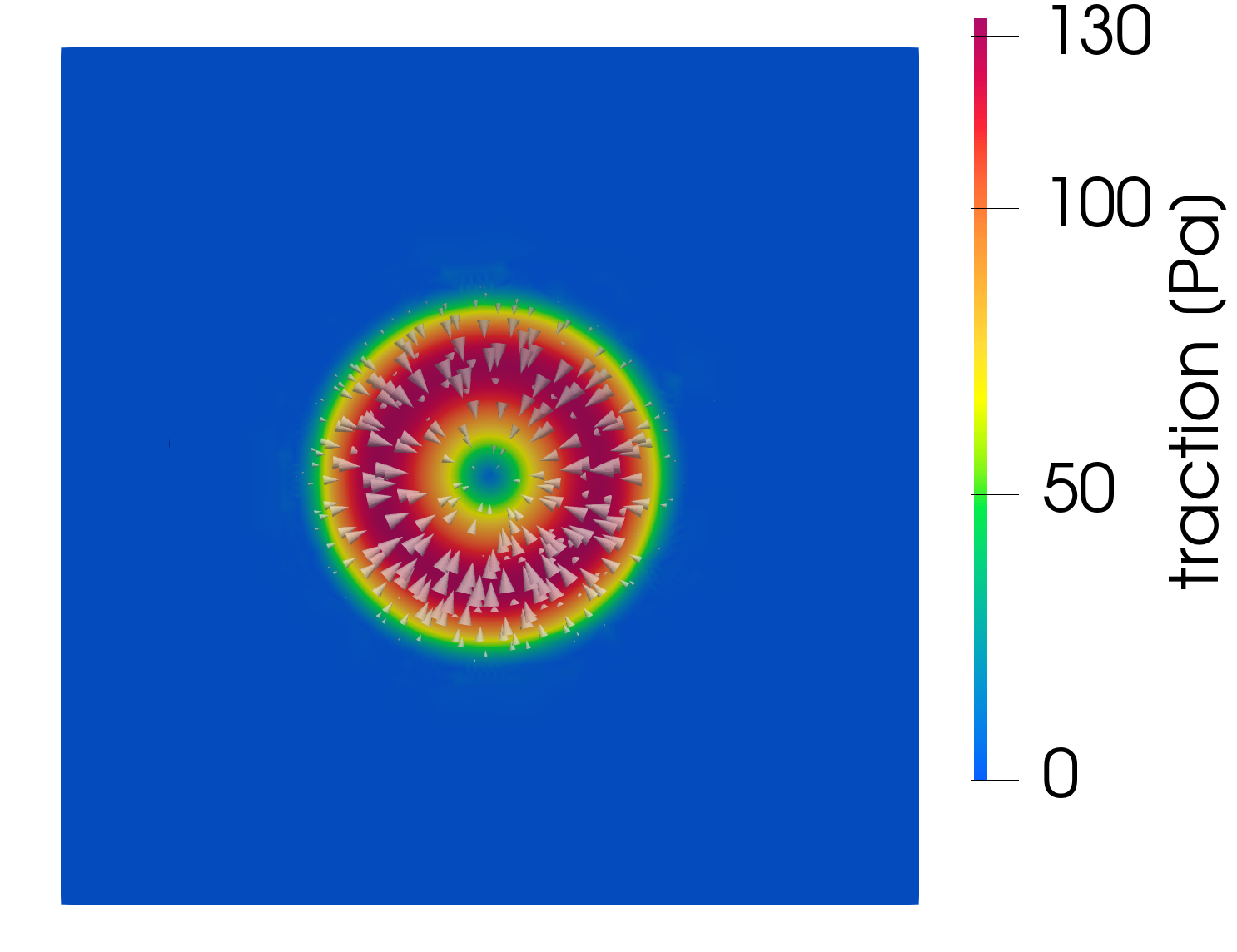}
\caption{Ground truth.}
\end{subfigure}%
\hfill
\begin{subfigure}{0.33\textwidth}
\centering
\includegraphics[width=\linewidth]{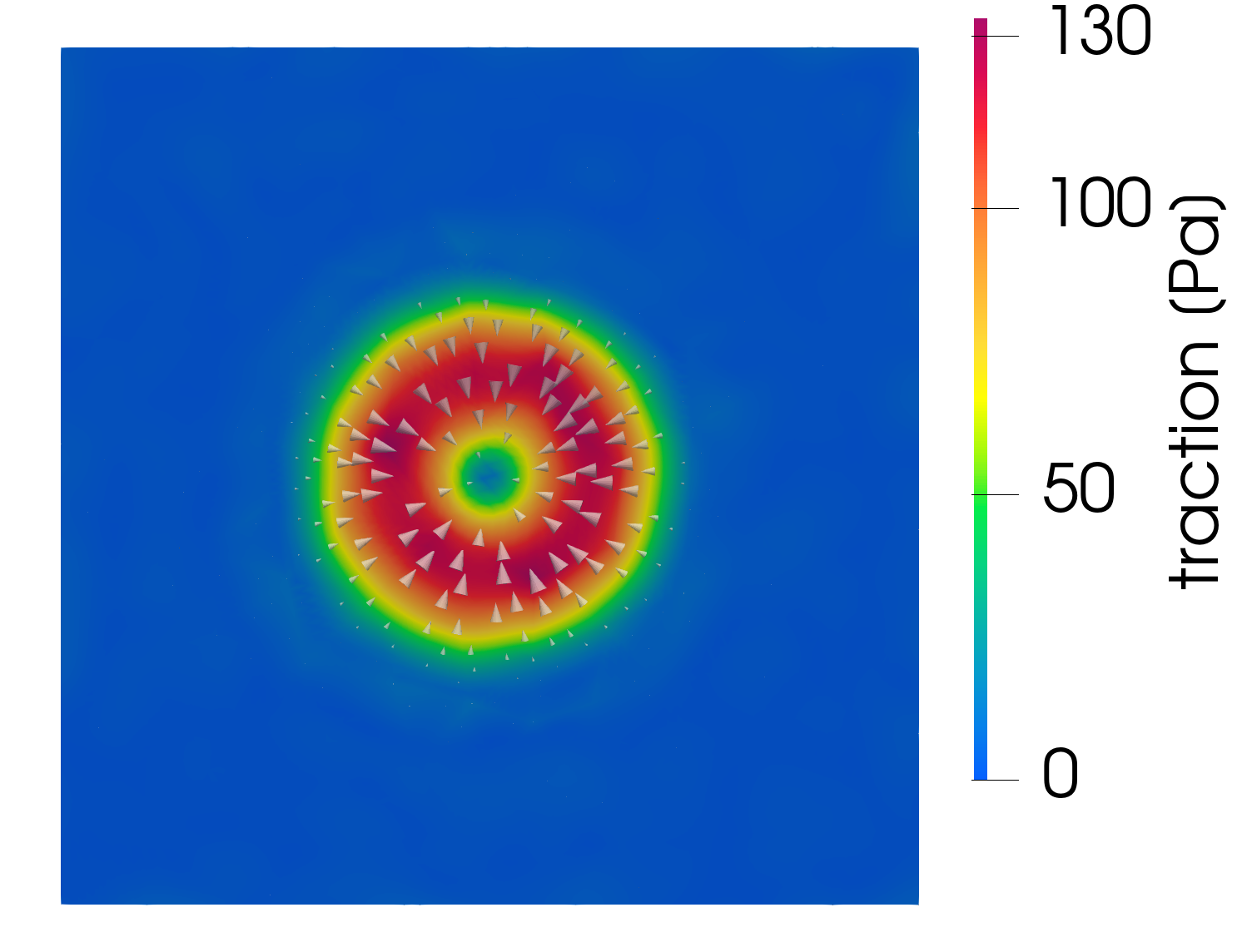}
\caption{Reconstruction.}
\end{subfigure}
\hfill
\begin{subfigure}{0.33\textwidth}
\centering
\includegraphics[width=\linewidth]{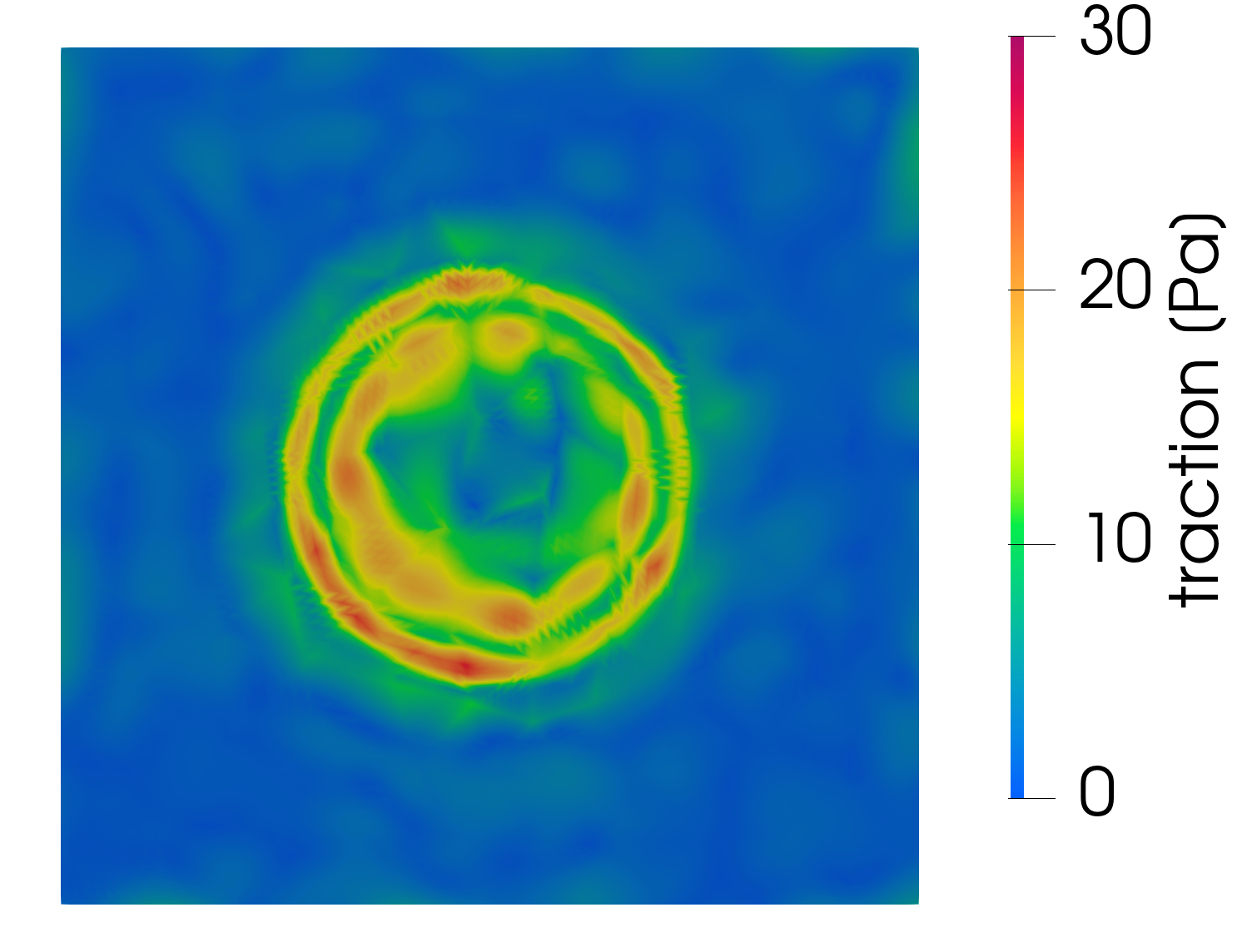}
\caption{Error.}
\end{subfigure}
\caption{2.5D linear TFM: simulated traction force via~\eqref{eq:force1}. The color code shows the magnitude of the traction stress (force per area), which is physically measured in Pa.}
\label{fig:3D_circ}
\end{figure}

Finally, we test our results on simulated data using the inverse problems python library `regpy'~\cite{regpy}. Further details can be found on Git~\url{https://github.com/regpy/regpy}. The TFM codes can be found on Gitlab at~\url{https://gitlab.gwdg.de/sarnighausen/traction-force-microscopy}. Evaluating the forward operator means solving a BVP. To this end, we use third order finite elements on an unstructured mesh from the finite element python library 'ngsolve'~\cite{ngsolve}. Details can also be found on the homepage~\url{https://ngsolve.org/}.
Then we generate the simulated noise-free data by applying the forward operators to a simulated traction stress. To prevent committing an inverse crime, we add Gaussian noise and use a different FEM mesh for reconstruction. Some basic examples of inverse crimes and how to avoid them can be found in \cite[Chapter II]{inverseCrime}.

\subsection{Simulated traction stresses}

\begin{figure}
\centering
\begin{subfigure}{0.33\textwidth}
\centering
\includegraphics[width=\linewidth]{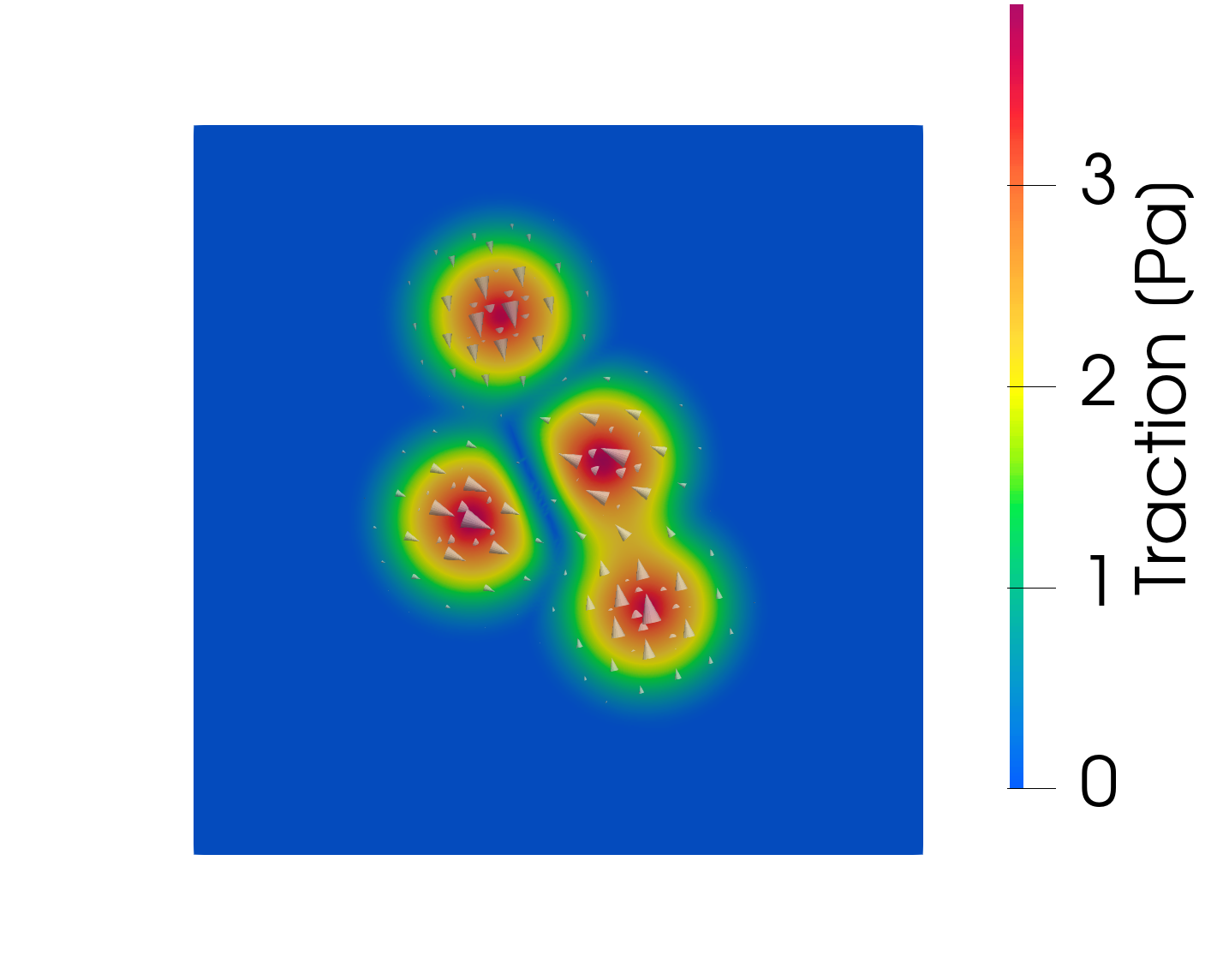}
\caption{Ground truth.}
\end{subfigure}%
\hfill
\begin{subfigure}{0.33\textwidth}
\centering
\includegraphics[width=\linewidth]{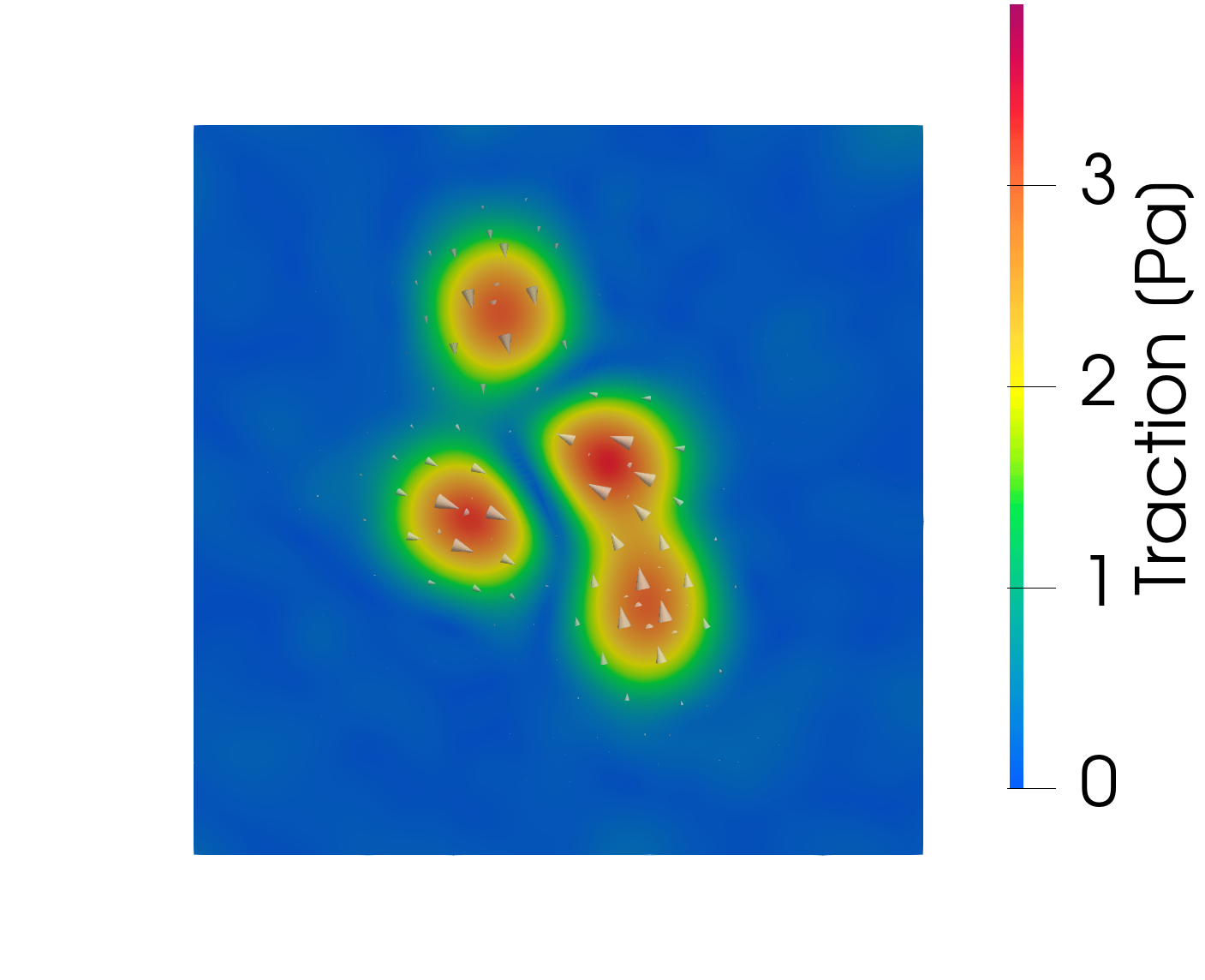}
\caption{Reconstruction.}
\end{subfigure}
\hfill
\begin{subfigure}{0.33\textwidth}
\centering
\includegraphics[width=\linewidth]{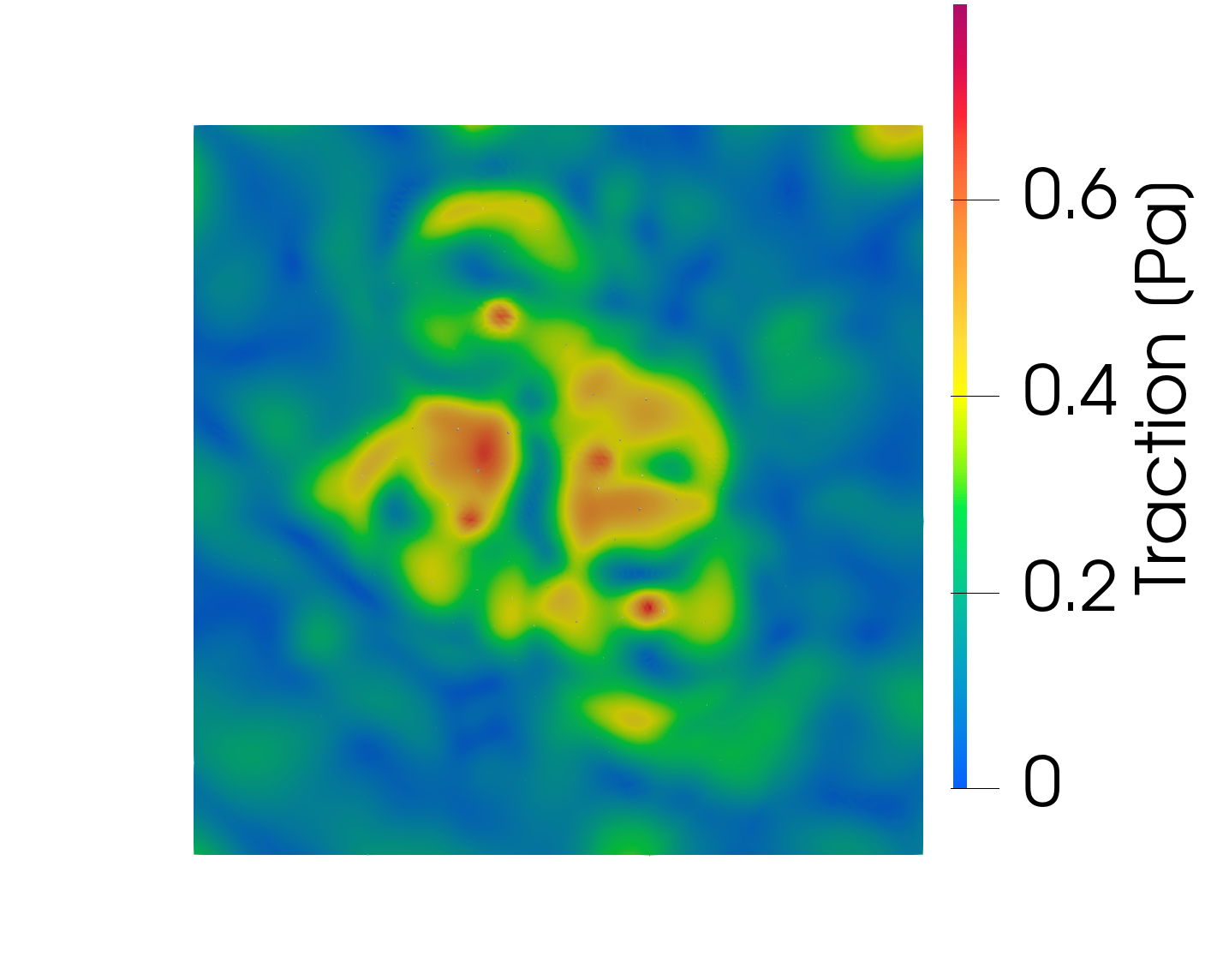}
\caption{Error.}
\end{subfigure}
\caption{2.5D linear TFM: simulated traction force as in~\eqref{eq:force2}.}
\label{fig:3D_cell}
\end{figure}

We test the reconstruction algorithms for two different force fields on a substrate with Young's modulus $\tilde{E}=10\, \mathrm{kPa}$ and Poisson ratio $\nu=0.45$ which are similar to experimental data. 
The first force field $t_1$ is a symmetric ring that pulls towards the middle, see Figure~\ref{fig:3D_circ}. It is defined by

\begin{align}
\label{eq:force1}
    t_1(x) =  \begin{cases}  a \cdot e^{\frac{-1}{1- (x_1^2 + x_2^2)}}  
    (-x_1,-x_2,0)^{\top}, & \lvert x \rvert \leq 1 \\
         0, & \text{else} \end{cases}
\end{align}
where the constant $a$ determines the magnitude of the force. We choose $a = 1000$. In a physical experiment, the unit of $a$ is given in Pascal (Pa). 

The second force is a more realistic simplification of a cell force. It consists of four force spots, see Figure~\ref{fig:3D_cell}, and is defined by
\begin{align}
\label{eq:force2}
t_2(x) = 10 \cdot e^{\frac{-1}{1- \lvert x - y^i \rvert^2 }} (d_1^i, d_2^i, 0)^{\top}, \quad \lvert x - y^i \rvert \leq 1, 
\quad\mbox{with}\quad 
    \begin{tabular}{ccccc}
        i & $y_1^i$ &  $y_2^i$ & $d_1^i$ & $d_2^i$\\
        \hline
        1 & -0.6  & -0.2 & 1 & -0.4 \\
        2 & 0.3 & 0.2 & -1 & 0.4 \\
        3 & 0.6 & -0.8 & -0.2 & 1 \\
        4 & -0.4 & 1.2 & 0.2 & -1 \\
    \end{tabular}
\end{align}

\begin{figure}[h]
\centering
\begin{subfigure}[t]{0.33\textwidth}
\centering
\includegraphics[width=\linewidth]{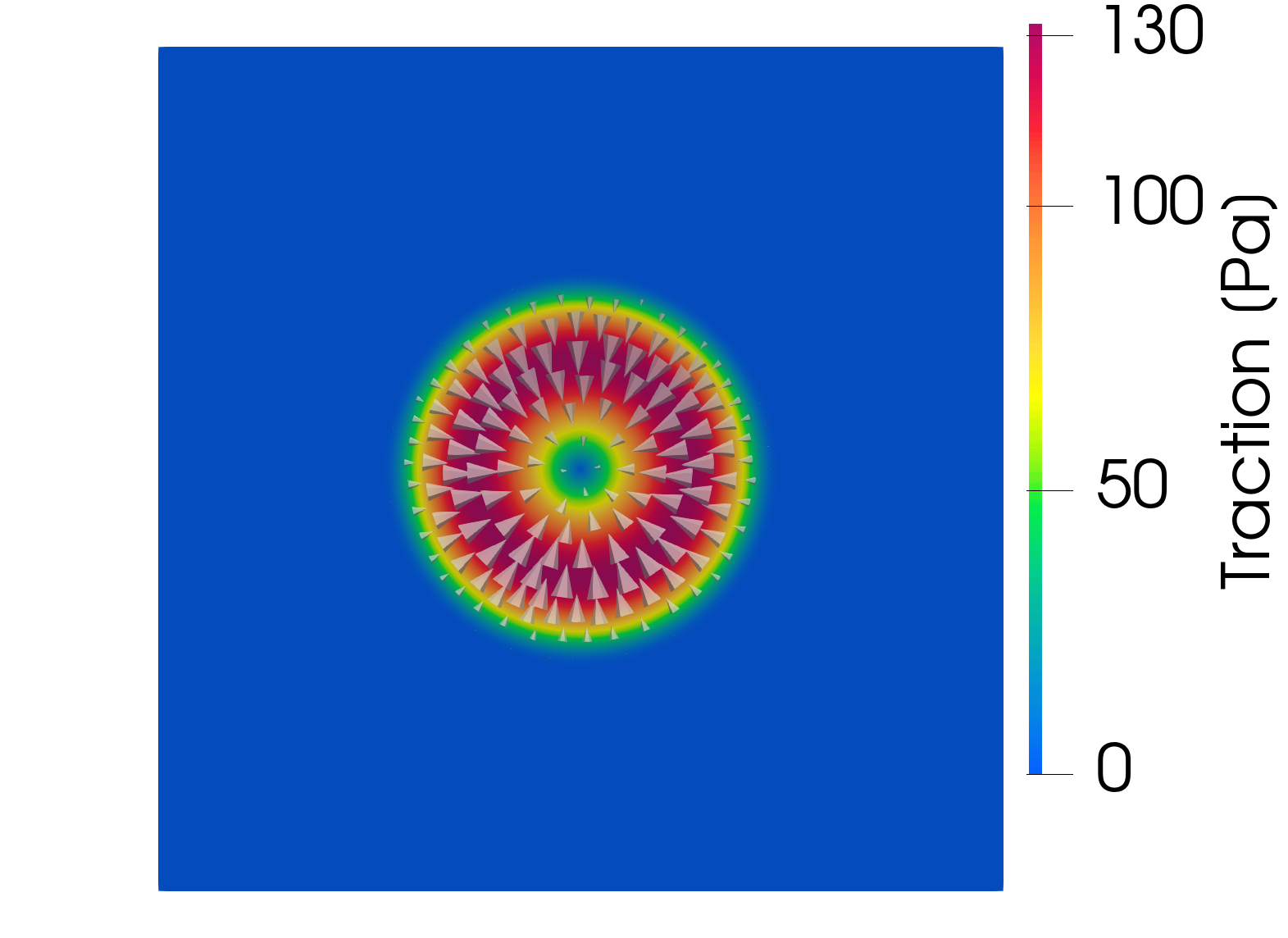}
\caption{Ground truth.}
\label{fig:2D_circ_true}
\end{subfigure}%
\begin{subfigure}[t]{0.33\textwidth}
\centering
\includegraphics[width=\linewidth]{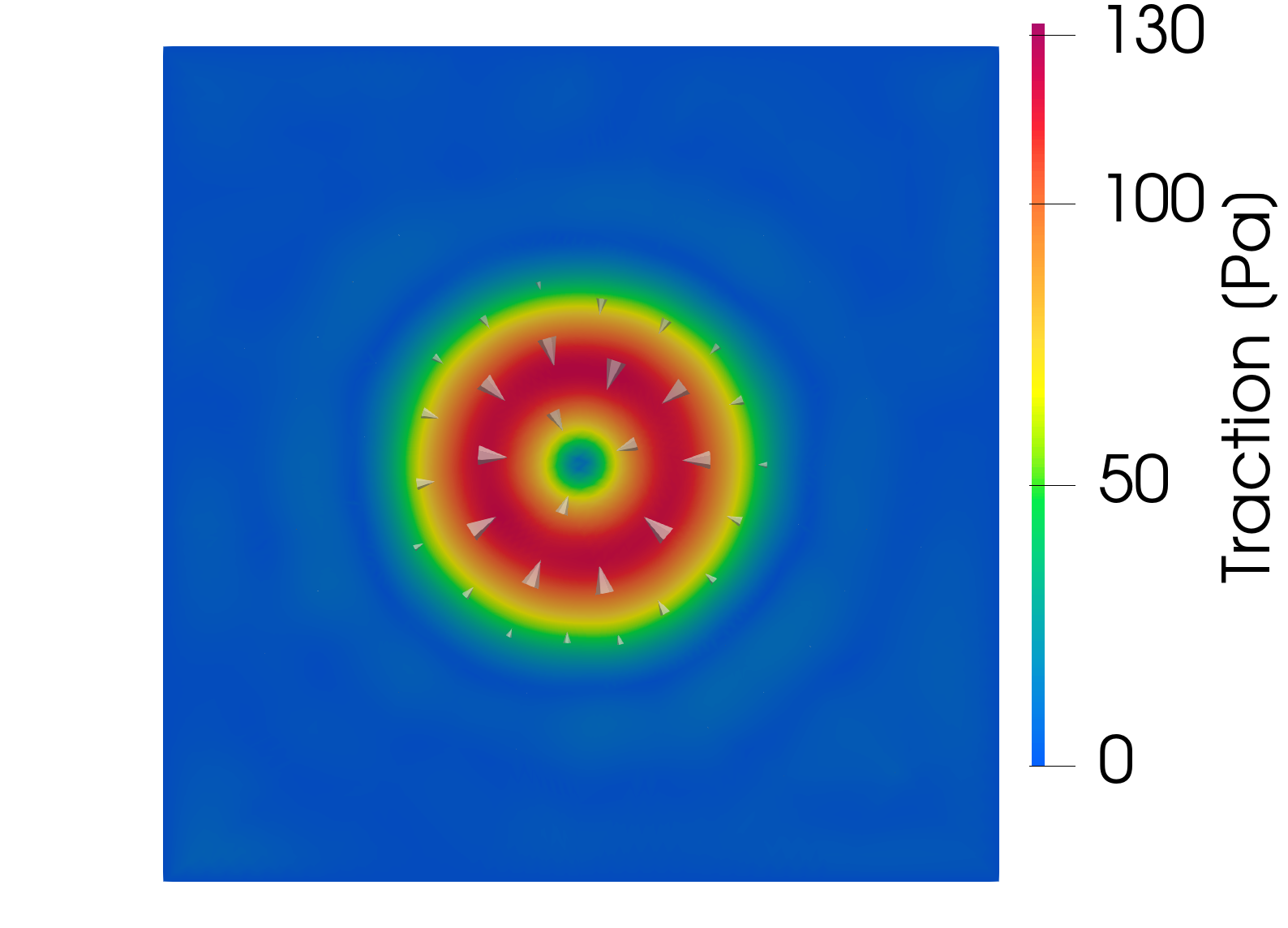}
\caption{Reconstruction $L^2$.}
\label{fig:2D_circ_rec_L2}
\end{subfigure}%
\begin{subfigure}[t]{0.33\textwidth}
\centering
\includegraphics[width=\linewidth]{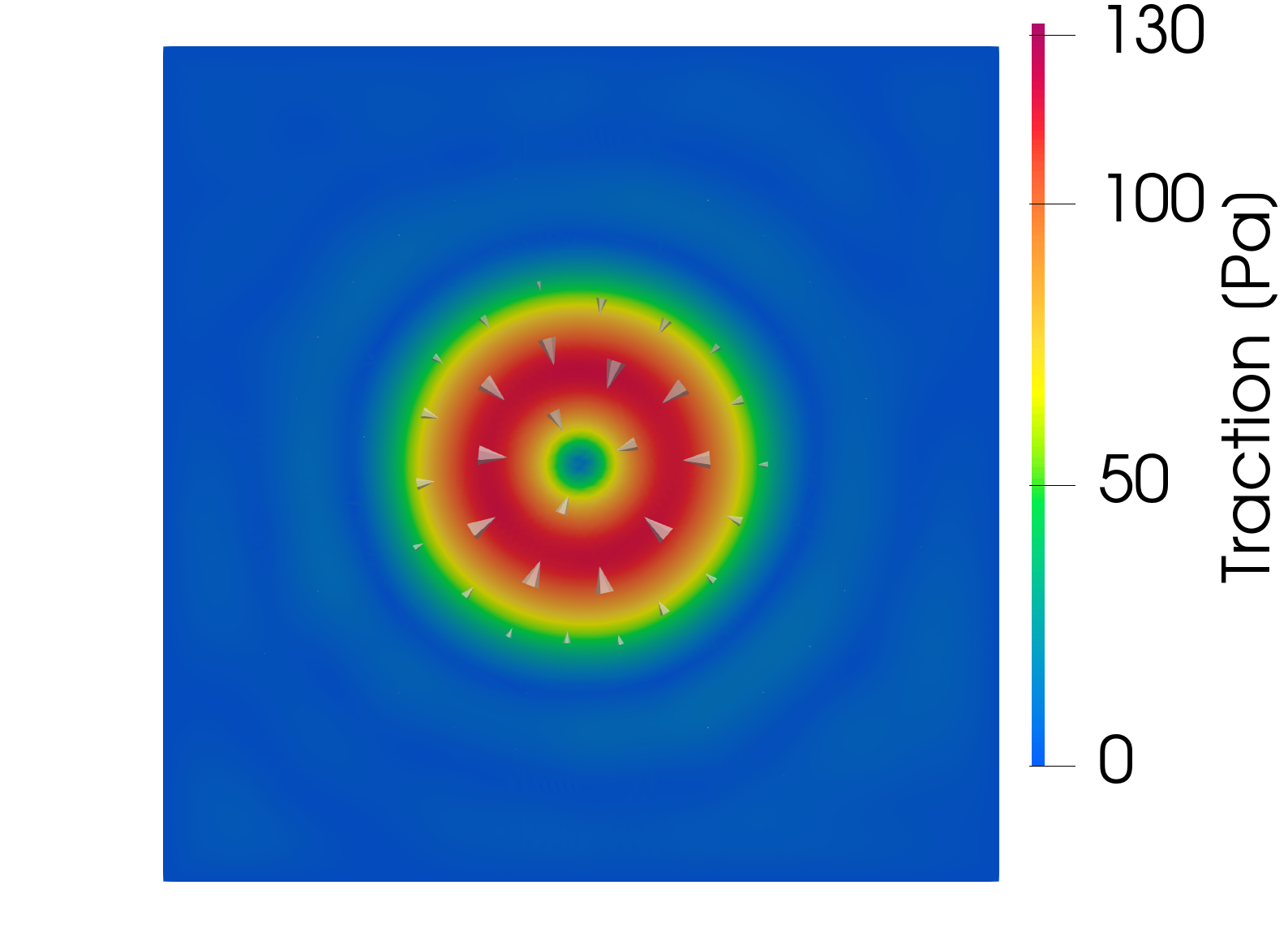}
\caption{Reconstruction $H^1_0$.}
\label{fig:2D_circ_rec_H01}
\end{subfigure}
\\
\begin{subfigure}[t]{0.33\textwidth}
\centering
\includegraphics[width=\linewidth]{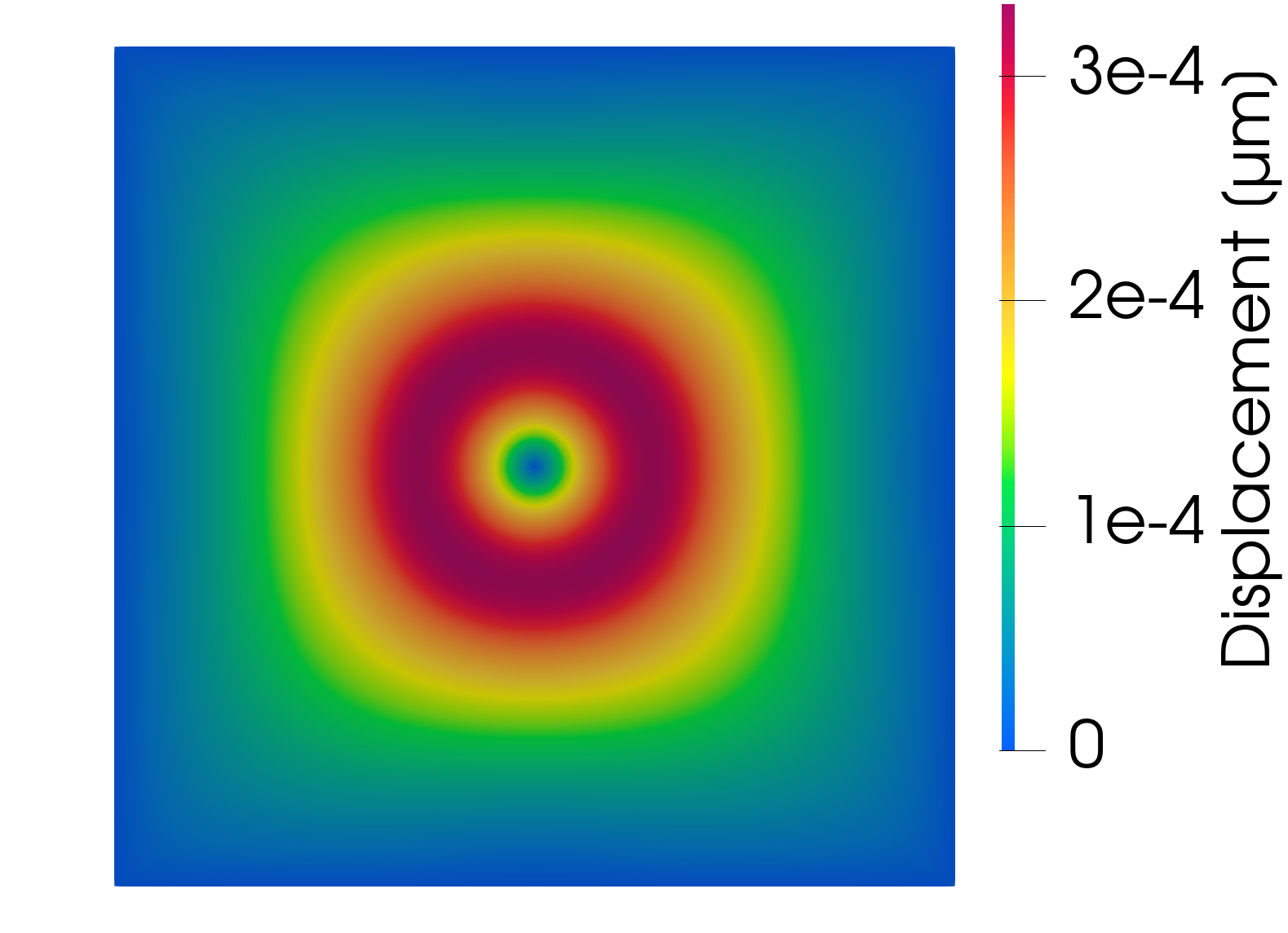}
\caption{Data.}
\label{fig:2D_circ_disp}
\end{subfigure}%
\begin{subfigure}[t]{0.33\textwidth}
\centering
\includegraphics[width=\linewidth]{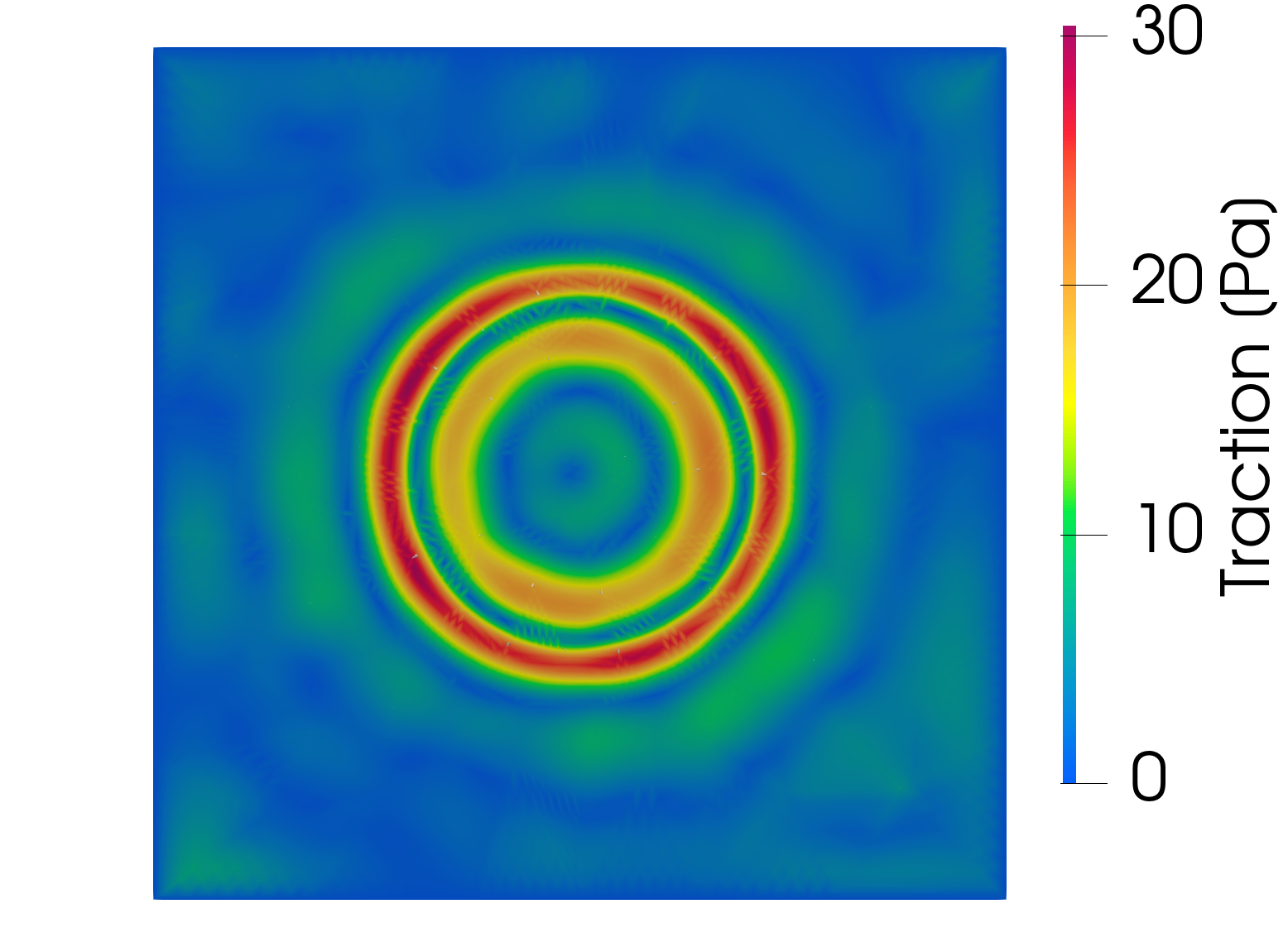}
\caption{Error $L^2$.}
\label{fig:2D_circ_err_L2}
\end{subfigure}%
\begin{subfigure}[t]{0.33\textwidth}
\centering
\includegraphics[width=\linewidth]{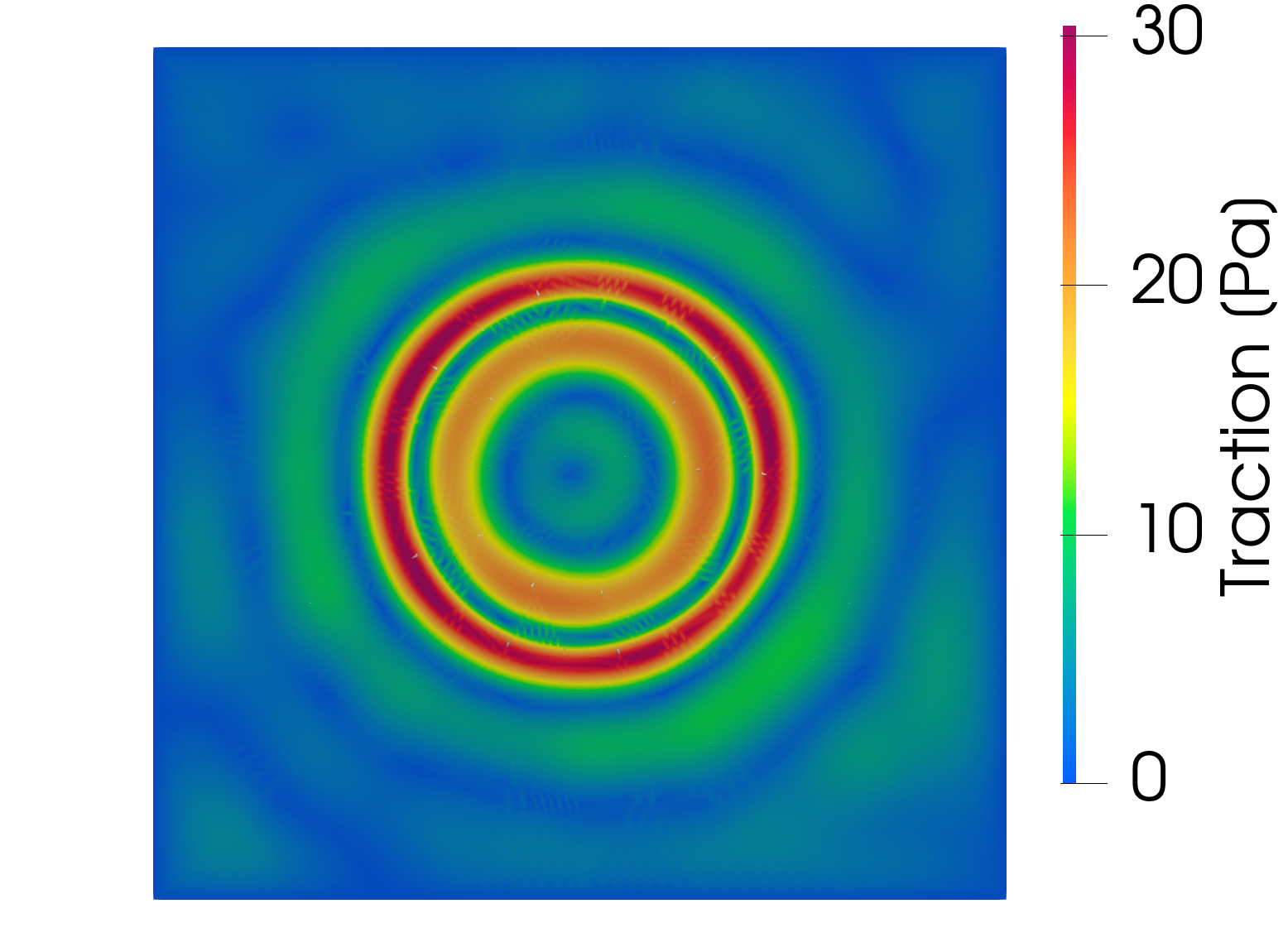}
\caption{Error $H^1_0$.}
\label{fig:2D_circ_err_H01}
\end{subfigure}
\caption{2D nonlinear TFM: simulated traction ground truth  as defined in \eqref{eq:force1} (a), reconstruction with $L^2$-penalty (b) and $H^1_0$-penalty (c), noise-free displacement data (d) and error with $L^2$-penalty (e) and $H^1_0$-penalty (f).}
\label{fig:2D_circ}
\end{figure}

\subsection{Linear 2.5D model}

To reconstruct the forces from the displacement data $u^\delta$, we use the conjugate gradient method applied to the normal equations
\begin{align}
\label{eq:normalEquations}
    \hat{A}^* \hat{A} t = \hat{A}^* u^\delta
\end{align}
(CGNE), see e.g.~\cite{Engl1996RegularizationOI}. As a stopping criterion we use the discrepancy principle, i.e. we stop the iteration the first time when
\begin{align}
\label{eq:discrepancy}
    \| \hat{A} t_k - u^\delta \| \leq \tau \delta
\end{align}    
is fulfilled for the forward operator $\hat{A}$, the current iterate $t_k$, the measured data $u^\delta$, the norm of the noise $\delta$ and a constant $\tau > 1$.

The CGNE method is a fast iterative method which is regularized by early stopping and can be used for solving linear inverse problems. The reconstruction and error are displayed in Figures \ref{fig:3D_circ} and~\ref{fig:3D_cell}. In Table~\ref{tab:3D} we see that we get good reconstructions for both forces even with a higher noise level. The run time for the force field~\eqref{eq:force2} is higher than for the force field~\eqref{eq:force1} due to a larger finite element matrix (degrees of freedom), even though the iteration is stopped earlier resulting from a higher noise level. 

\begin{table}[h]
    \centering
    \begin{tabular}{|c|c|c|c|c|c|c|}
    \hline
         & $L^2$-error & noise level & run time & iterations & ndof & $\tau$ \\
    \hline
      force field~\eqref{eq:force1}   & $15.91\%$ & $5\%$ & $12.68$ s & 30 & 15 552 & 1.2\\
    \hline
      force field~\eqref{eq:force2}   & $23.65 \%$ & $9.97 \%$ & $25.32$ s & 10 & 28 320 &1.01\\
    \hline
    \end{tabular}
    \caption{Reconstruction error, noise level, run time, number of CGNE iterations and degrees of freedom of the finite element space for both forces.}
    \label{tab:3D}
\end{table}

\subsection{Nonlinear pure 2D model}
\label{subsec:numericsNonlinear2D}

To numerically evaluate the nonlinear forward mapping, 
the deformation field is found by the minimizing the stored energy~\eqref{eq2:storedEnergy} with Newton's method (see~\url{https://ngsolve.org/}). 
For solving the nonlinear inverse problem, we use the truncated Newton CG method, see~\cite[Section 2]{Hanke1997RegularizingPO}. After linearizing the forward operator, the Newton equation $S'(T_k)h_k = u^\delta - S(T_k)$ with noisy measured displacement $u^\delta$ for the $k$-th iteration is solved with the CGNE method for the normal equation $S'(T_k)^* S'(T_k)h_k = S'(T_k)^* \bigl( u^\delta - S(T_k) \bigr)$ (inner iteration) with $S'(T_k), S'(T_k)^*$ as described in Section \ref{sec2:Frechet}, then the force density is updated by $T_{k+1} = T_k + h_k$ (outer iteration). The inner iteration is stopped early when the residual of the inner linearized problem is reduced to a percentage $0 < \rho < 1$ of the residual of the outer nonlinear problem. Here we choose $\rho = 0.7$. We stop the outer iteration again with the discrepancy principle as described in Equation~\eqref{eq:discrepancy}.

The function space setting from Remark~\ref{rem:funcSpace} yields $X = H_0^1(\Omega,\R^2)$ and $Y = L^2(\Omega,\R^2)$. However, numerically we also try the choice $X = L^2(\Omega,\R^2)$.
Depending on the choice of $X$, we use
different norms when computing the solution of the normal equations in the CGNE method, leading to different results. 

The results for the first force field~\eqref{eq:force1} are summarized in Figure~\ref{fig:2D_circ} and Table~\ref{tab:err_circ_2D} and for the second force field~\eqref{eq:force2} in Figure~\ref{fig:2D_cell} and Table~\ref{tab:err_cell_2D}. For both force fields we choose the effective thickness $h=1$ and the tolerance $\tau = 1.01$ in the discrepancy principle.

\begin{table}[h]
    \centering
    \begin{tabular}{|c|c|c|c|}
    \hline
         & error  & run time & iterations \\
         \hline 
        $L^2$ & $\mathbf{20.57\%}$ & $10.17$s  & 9 \\
        $H^1_0$ & $22.86\%$ & $60.11$s & 37 \\
        \hline
    \end{tabular}
    \caption{Reconstruction error, run time and number of Newton CG iterations of nonlinear 2D TFM for the first force~\eqref{eq:force1} with a noise level of $3.54 \%$ and $2048$ degrees of freedom of the finite element space.}
    \label{tab:err_circ_2D}
\end{table}

\begin{table}[h]
    \centering
    \begin{tabular}{|c|c|c|c|c|}
        \hline
         noise level & & error & run time & iterations \\
         \hline 
         $7.81\%$ &  $L^2$ & $\mathbf{51.16\%}$ & $5.25 $s &  $6$ \\
         & $H^1_0$ & $51.27\%$  & $10.36$s & $7$  \\
         \hline
         $15.63\%$ & $L^2$ & $ 82.93\%$  & $7.21$s  & $5$ \\
         &  $H^1_0$ & $\mathbf{55.92\%}$  & $11.48$s & $6$ \\
         \hline
    \end{tabular}
    \caption{Reconstruction error, run time and number of Newton CG iterations of nonlinear 2D TFM for the second force~\eqref{eq:force2} with $5240$ degrees of freedom of the finite element space.}
    \label{tab:err_cell_2D}
\end{table}

The noise level (in \%) is calculated by 
\begin{align}
\label{eq2:noiselevel}
    \text{noise level} = \frac{\lVert \text{noise} \rVert_{L^2(\Omega)}}{\lVert \text{exact data} \rVert_{L^2(\Omega)} } \cdot 100
\end{align}
where we use Gaussian noise 
in accordance with the experiment, see~\cite{Plotnikov2014, schwarzComparison}.

\begin{figure}[h]
\centering
\begin{subfigure}[t]{0.3\textwidth}
\centering
\includegraphics[width=\linewidth]{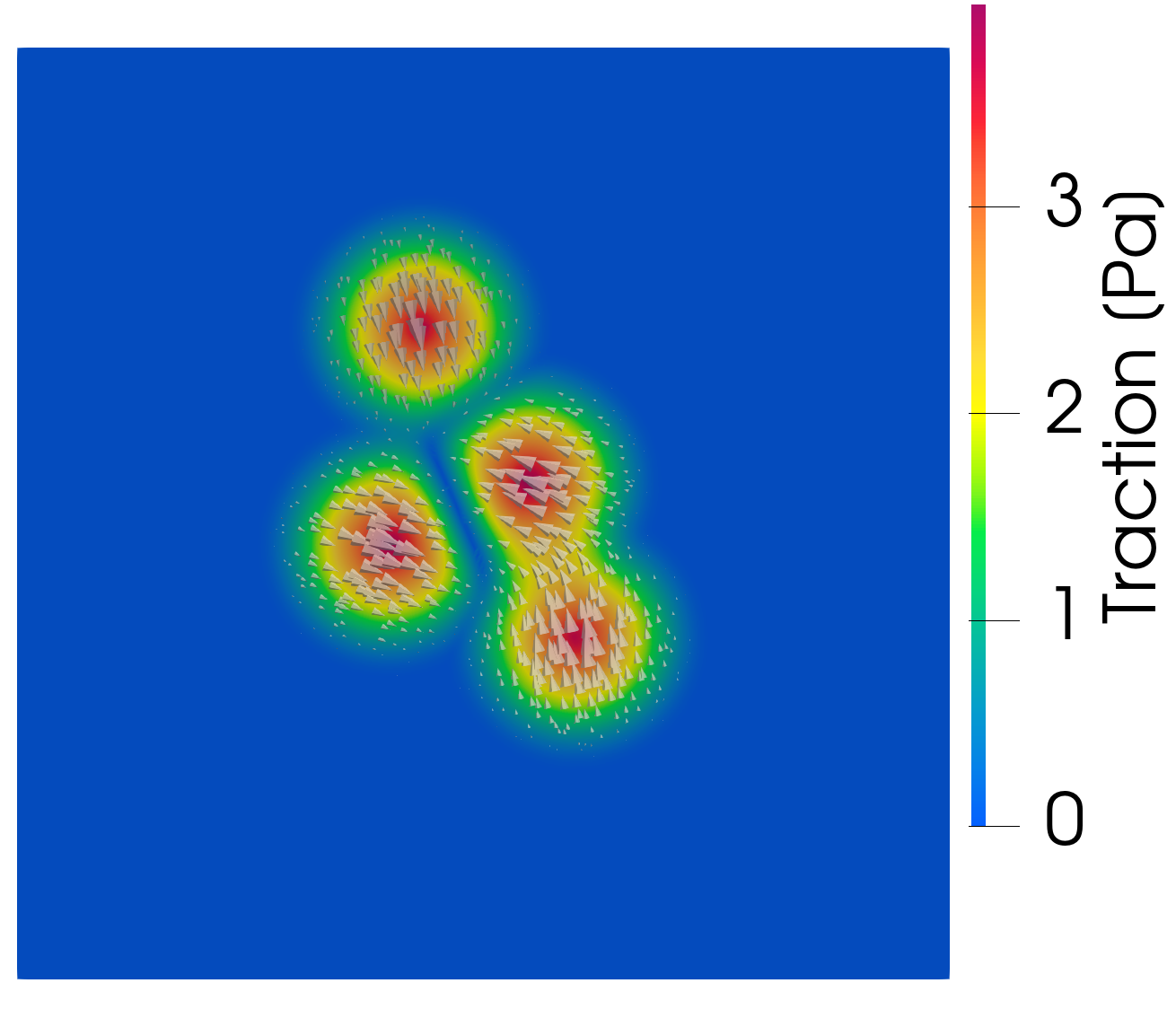}
\caption{Ground truth.}
\label{fig:2D_cell_true}
\end{subfigure}%
\vspace*{0.03\textwidth}
\begin{subfigure}[t]{0.33\textwidth}
\centering
\includegraphics[width=\linewidth]{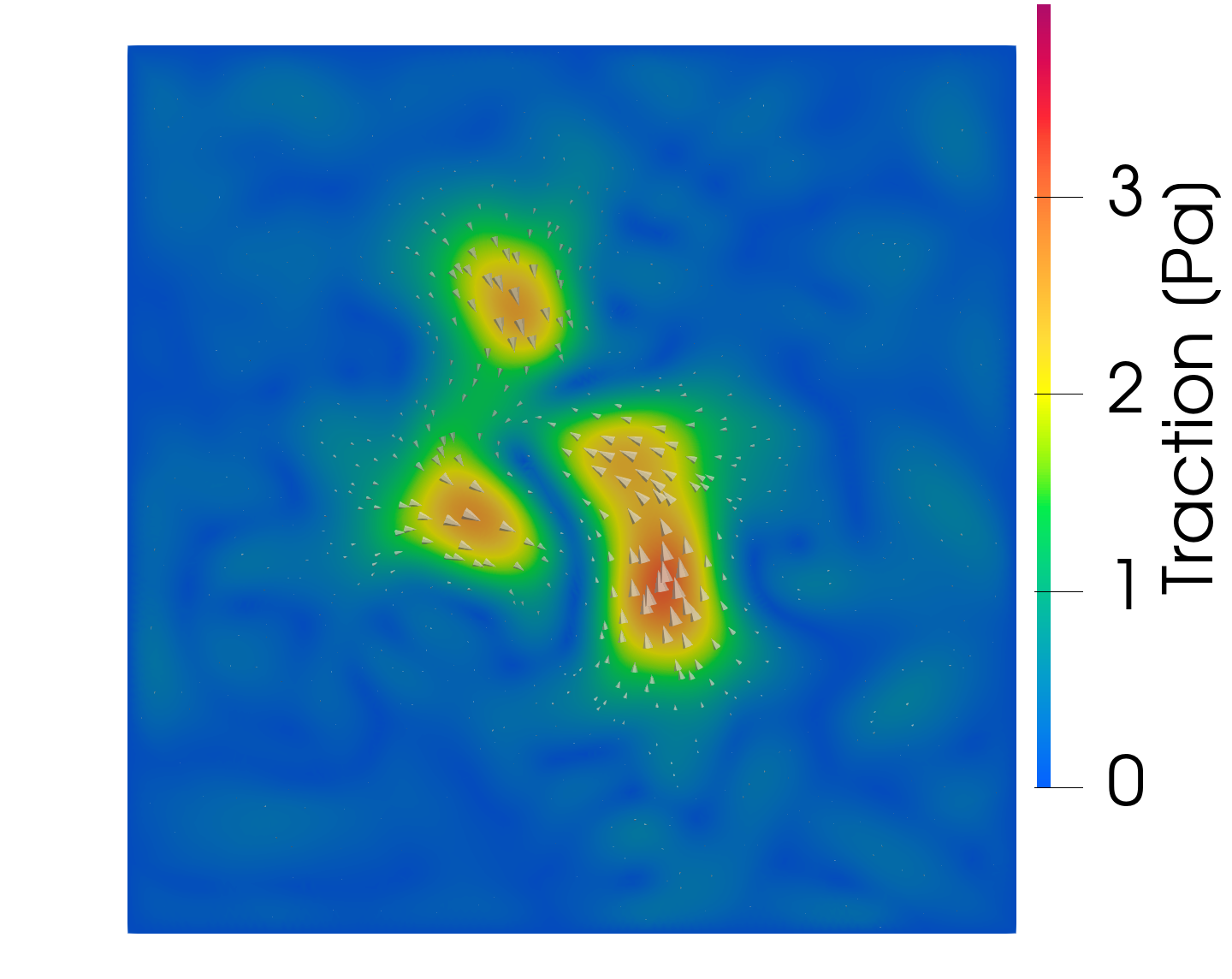}
\caption{Reconstruction $L^2$, low noise.}
\label{fig:2D_cell_small_L2}
\end{subfigure}%
\begin{subfigure}[t]{0.33\textwidth}
\centering
\includegraphics[width=\linewidth]{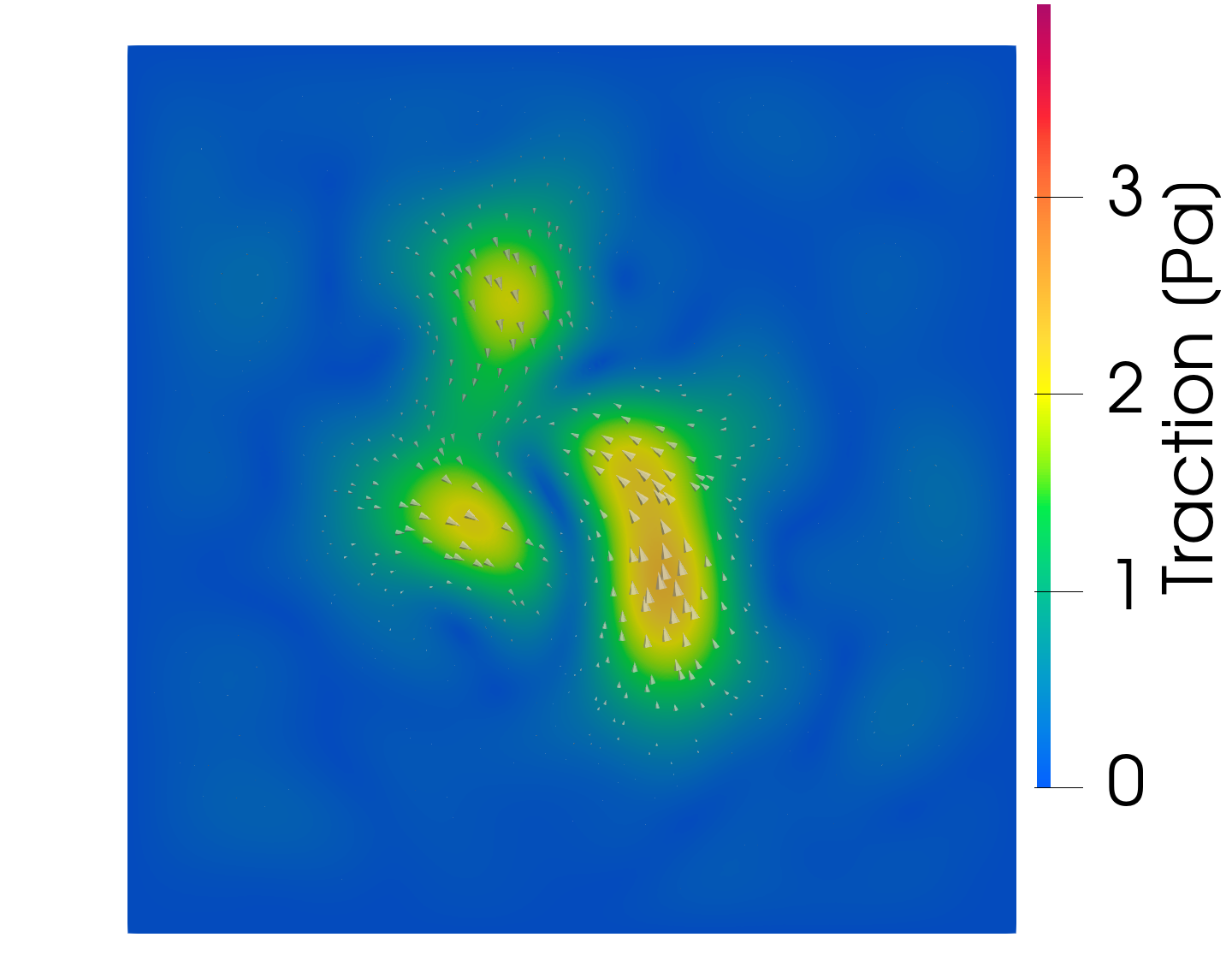}
\caption{Reconstruction $H^1_0$, low noise.}
\label{fig:2D_cell_small_H01}
\end{subfigure}
\\
\begin{subfigure}[t]{0.33\textwidth}
\centering
\includegraphics[width=\linewidth]{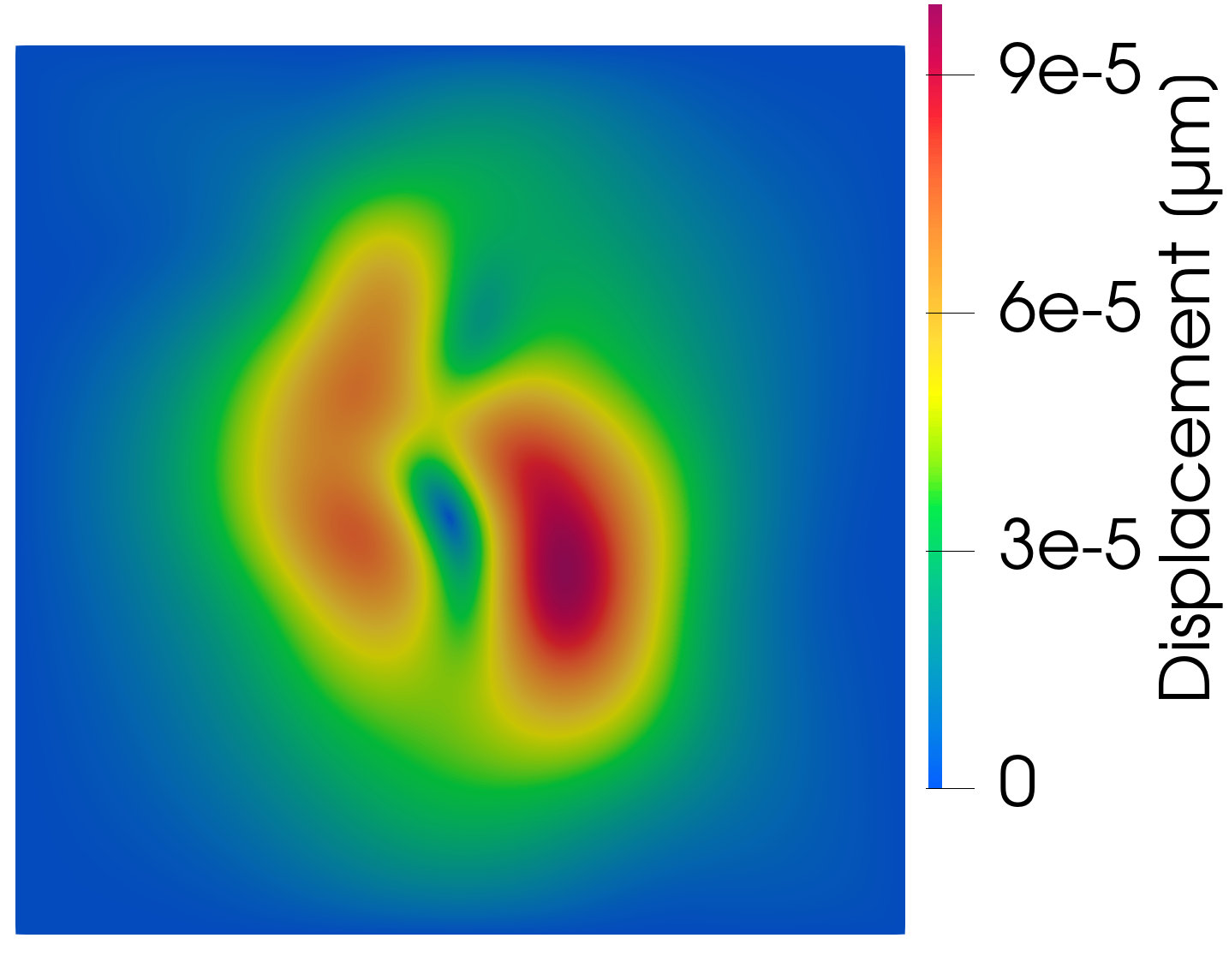}
\caption{Data.}
\label{fig:2D_cell_displacement}
\end{subfigure}%
\begin{subfigure}[t]{0.33\textwidth}
\centering
\includegraphics[width=\linewidth]{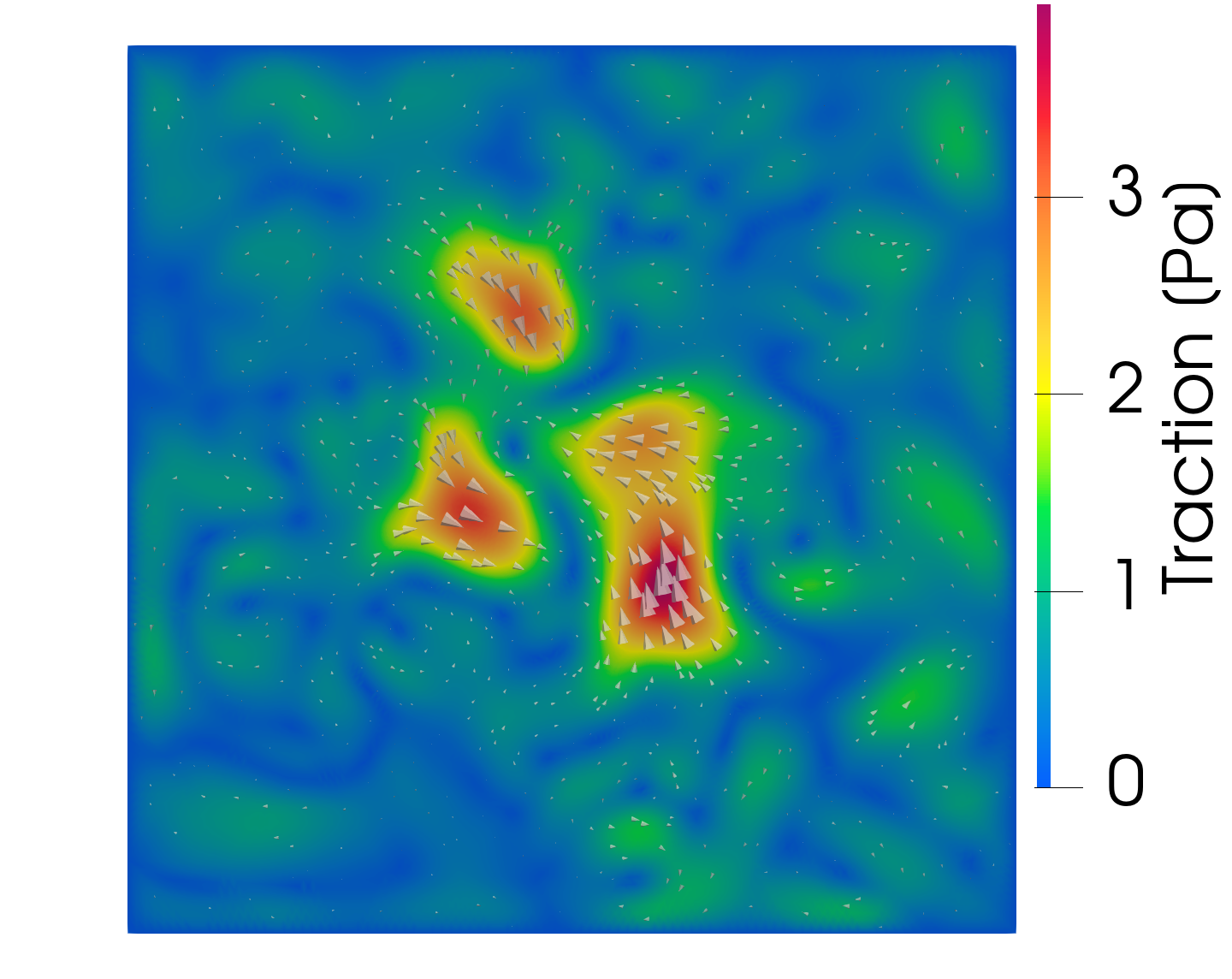}
\caption{Reconstruction $L^2$, high noise.}
\label{fig:2D_cell_big_L2}
\end{subfigure}%
\begin{subfigure}[t]{0.33\textwidth}
\centering
\includegraphics[width=\linewidth]{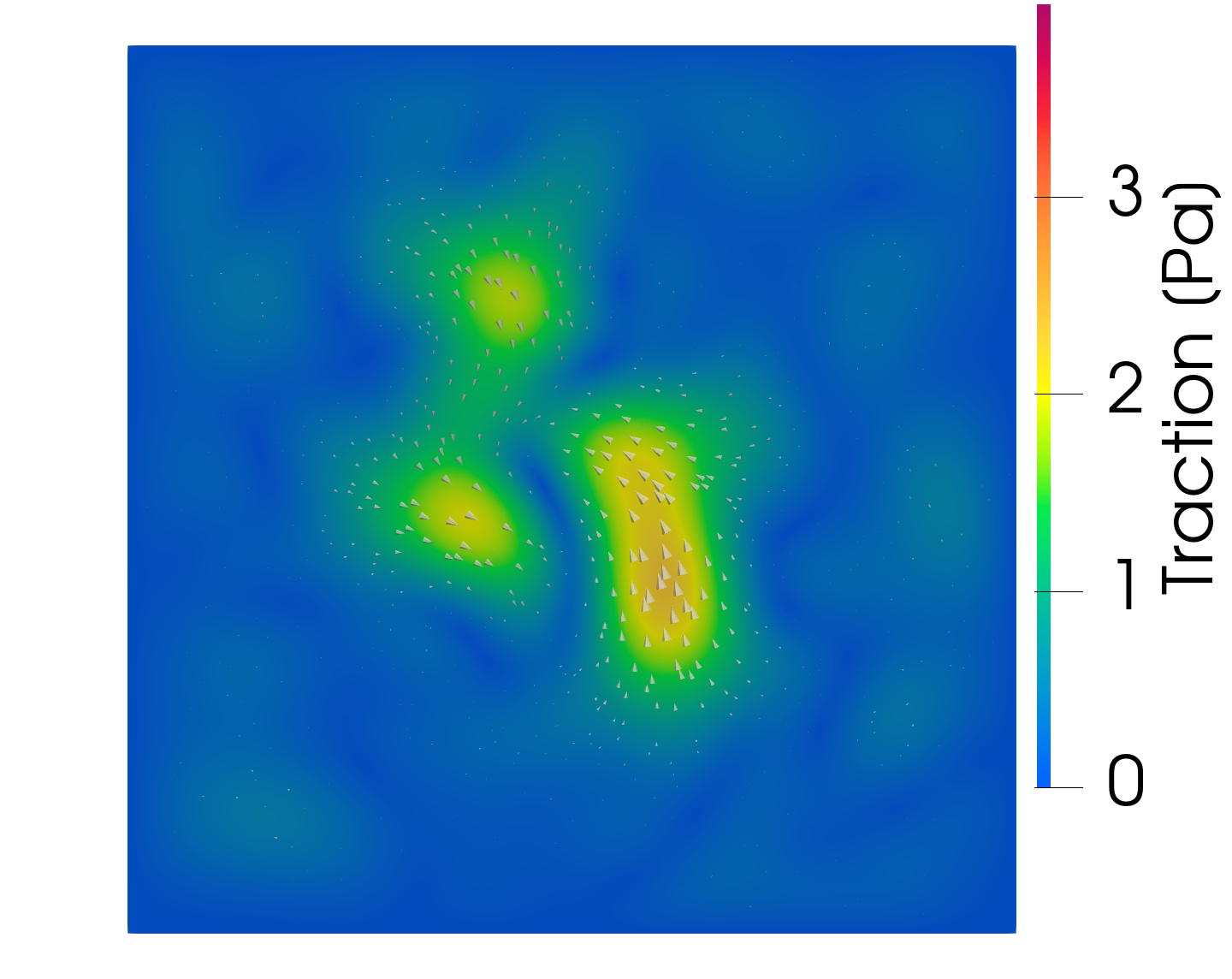}
\caption{Reconstruction $H^1_0$, high noise.}
\label{fig:2D_cell_big_H01}
\end{subfigure}
\caption{2D nonlinear TFM: simulated traction ground truth as in \eqref{eq:force2} (a), reconstruction with noise level of $7.81\%$ with $L^2$-penalty (b) and $H^1_0$-penalty (c), noise-free displacement data (d)  reconstruction with noise level of $15.62\%$ with $L^2$-penalty (e) and $H^1_0$-penalty (f).}
\label{fig:2D_cell}
\end{figure}

Even though the $L^2$-penalty does not incorporate the additional requirements on the traction forces $t$, the $L^2$-error is smallest for not too noisy data. However, this changes when the noise level gets larger. Then the $H^1_0$-penalty leads to the best result.

Using the $L^2$-penalty we get strong artifacts with larger noise, see Figures~\ref{fig:2D_cell_small_L2}, \ref{fig:2D_cell_big_L2}, whereas the noise gets smoothed out when using the $H^1_0$-penalty, see Figures~\ref{fig:2D_cell_small_H01}, \ref{fig:2D_cell_big_H01}. The gradient included in the $H^1_0$-norm could be the reason for this effect.

\subsection{Application to experimental data and comparison with a standard physics method}
\label{ssec:linearRealData}

In the end, we apply our method to measured linear 2D displacement data from a standard TFM experiment and compare it with the standard reconstruction method from physics, see e.g.~\cite{Schwarz2015}. 

For the TFM experiment, we plate NIH 3T3 fibroblast
cells on polyacrylamide (PAA) gels with a Young's
modulus of $34 \pm 1.6$ kPa. Fluorescently labeled carboxylated beads are added to the PAA solution as fiducial markers before it is polymerized on a treated glass bottom dish,
as described in detail, e.g., in~\cite{ZELENA2023, Hanke2018}. After polymerization, we coat the PAA gel with Sulfo-SANPAH
 and fibronectin, acting as an extracellular matrix (ECM) protein. TFM measurements are performed using an Olympus IX83 microscope under controlled physiological conditions.
The displacements are computed using a Kanade-Lucas-Tomasi (KLT) optical flow algorithm by comparing images of fiducial markers
of gels with cells and
of relaxed
gels. Then, traction forces are calculated using Fourier Transform Traction Cytometry (FTTC) with Tikhonov regularization~\cite{schwarzComparison}.

Since our purely 2D model is based on the older TFM model from~\cite{Dembo1996} that assumes thin substrates while the experiment is conducted under the assumptions of a not purely 2D model using thick substrates as in~\cite{Dembo1999}, it is not surprising that we observe differences in the reconstructions. 
However the results are clearly similar, see Figure~\ref{fig:real_data}.
Since our reconstructed traction has the same data range as the FTTC solution, our method also indicates that we indeed have an effective thickness $h \approx 1 \,\mu\mathrm{m}$ in this experiment.

To estimate the noise, we chose an area of the measured displacement that appears to have almost no displacement and therefore lies outside the cell boundary. Then we compute the $L^2$-norm on this area and scale the computed $L^2$-norm to the whole domain. This results in a relative noise level of approximately $18\%$. Here, the discrepancy principle stops the iteration after 90 steps with $\tau = 1.1$ in the discrepancy principle.

\begin{figure}
\centering
\begin{subfigure}{0.33\textwidth}
\centering
\includegraphics[width=\linewidth]{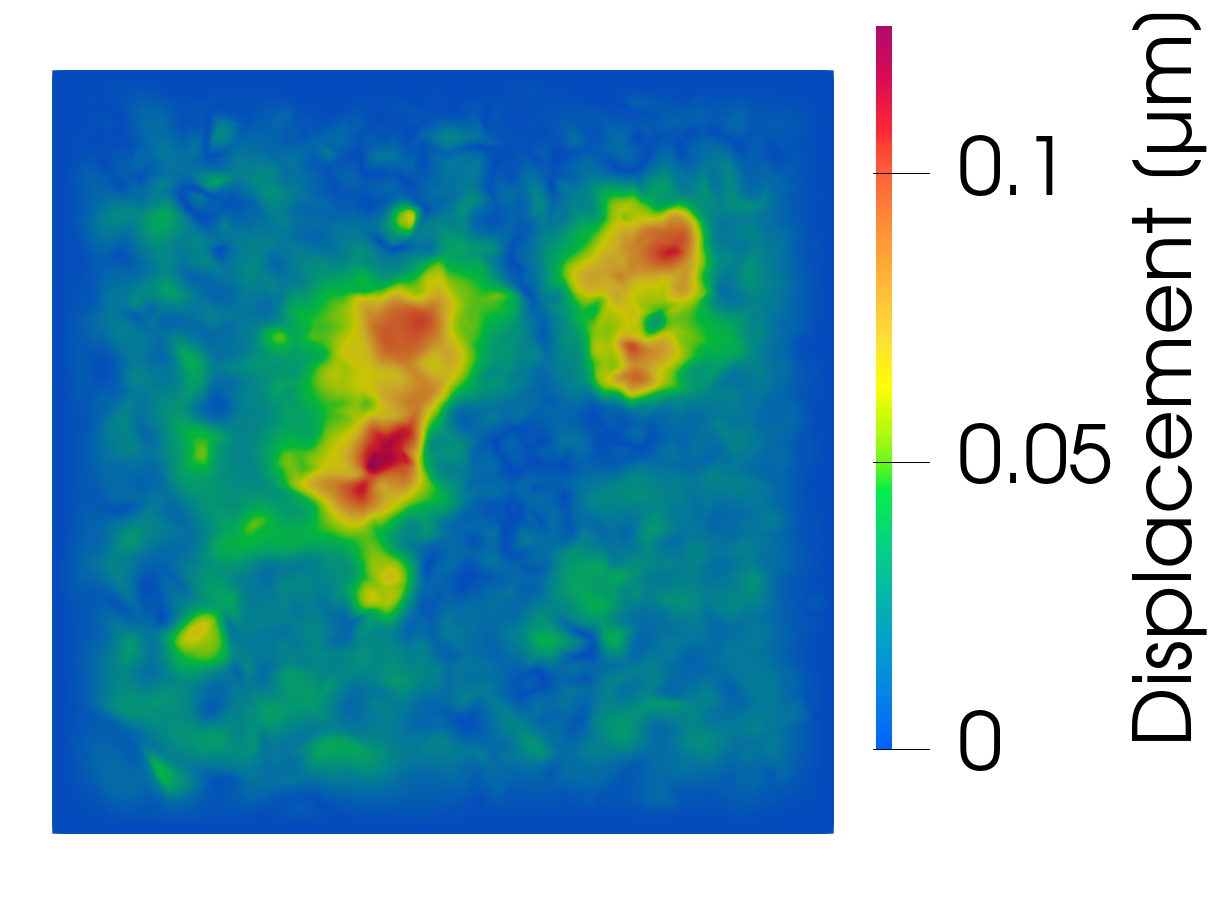}
\caption{Measured data.}
\end{subfigure}%
\hfill
\begin{subfigure}{0.33\textwidth}
\centering
\includegraphics[width=\linewidth]{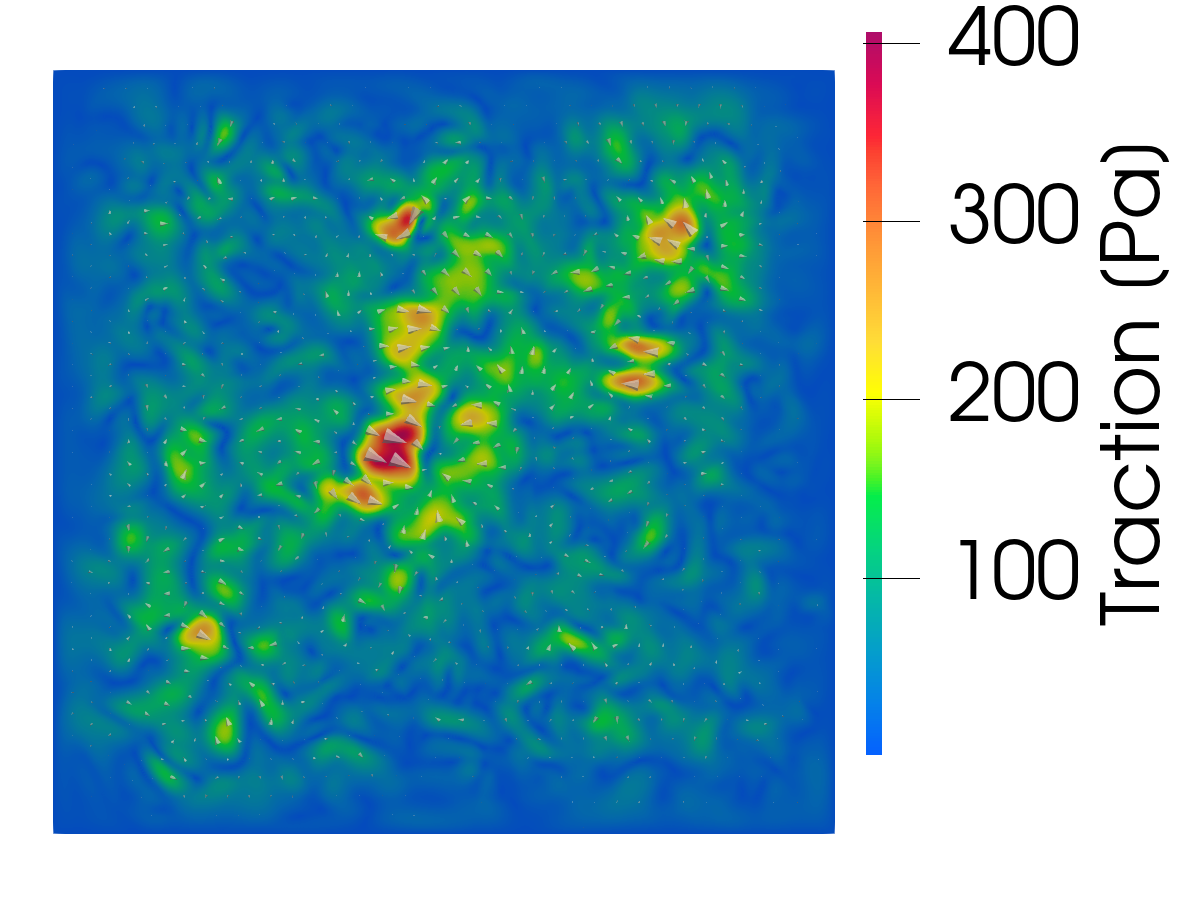}
\caption{Reconstruction.}
\end{subfigure}
\hfill
\begin{subfigure}{0.33\textwidth}
\centering
\includegraphics[width=\linewidth]{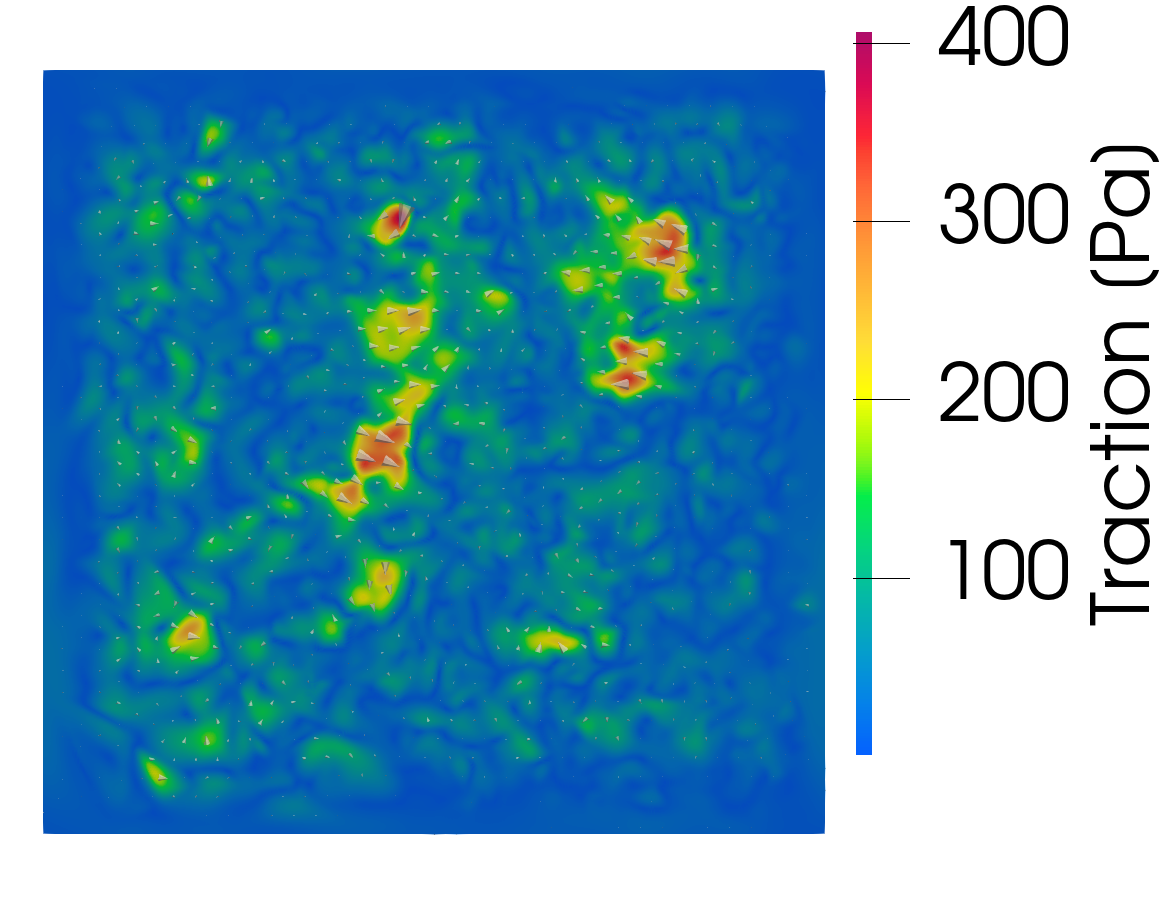}
\caption{FTTC.}
\end{subfigure}
\caption{2D linear TFM: real data (a), our reconstruction (b), standard physics method (c).}
\label{fig:real_data}
\end{figure}

\subsection{Comparison of pure linear and nonlinear 2D models}

To compare the pure nonlinear with the pure linear 2D model, we use the same traction force fields and compute the displacement using our nonlinear/linear forward model. For the linear 2D model, we use the 2D model~\eqref{eq2:tfmForward} with the linearized strain tensor and Hooke's law instead of the nonlinear constitutive equation~\eqref{eq2:constitutiveEquation}.

We use the force field described in equation~\eqref{eq:force1} with multiple values for the constant $a$ and compute the relative error in the $L^2$-norm of both displacements. 
Instead of using the absolute $L^2$-norm of the forces, we use the relative $L^2$-norm by dividing the absolute norm by the norm of the function that is constantly one on the domain.
The results are shown in Figure~\ref{fig:comparison_plot}.
This behavior motivates why mostly linear TFM is used in physics. As an example, the relative norm of the force reconstructed with FTTC (Figure~\ref{fig:real_data}) is $48.35$, where we observe a good correlation of the results from the linear and nonlinear model. This gives an indication that it is sufficient to use the linear model in this case.

As a further comparison, we generate noise-free data with the nonlinear model and force field~\eqref{eq:force1} with $a = 2 \cdot 10^5$ which corresponds to a relative norm of $5.39 \cdot 10^3$. If we compute a reconstruction using the linear model, we get a minimal error of $7.27 \%$ after $300$ CGNE iterations. Using the linear reconstruction as an initial guess, we can improve the reconstruction with the nonlinear model and get a minimal error of $2.58\%$ after $10$ Newton CG iterations, see Figure~\ref{fig:comparison_err_lin} and~\ref{fig:comparison_err_nl}.
Since the force has quite high values, we use the homotopy method with $10$ equidistant steps to improve the stability of the nonlinear forward solver. An example with explanations can be found in Section 3.4 of the ngsolve documentation.

\begin{figure}
    \centering
    \begin{minipage}[t]{0.6\textwidth}
        \centering
        \begin{subfigure}[t]{\textwidth}
        \centering
        \includegraphics[width=\linewidth]{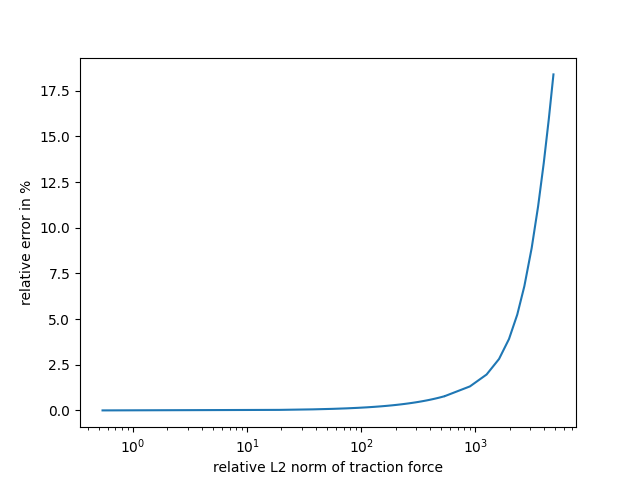}
        \caption{Comparison of linear and nonlinear forward model.}
        \label{fig:comparison_plot}
        \end{subfigure}
    \end{minipage}
    \hfill
    \begin{minipage}[t]{0.35\textwidth}
        \vspace{-45ex}
        \centering
        \begin{subfigure}[t]{\textwidth}
            \centering
            \includegraphics[width=\linewidth]{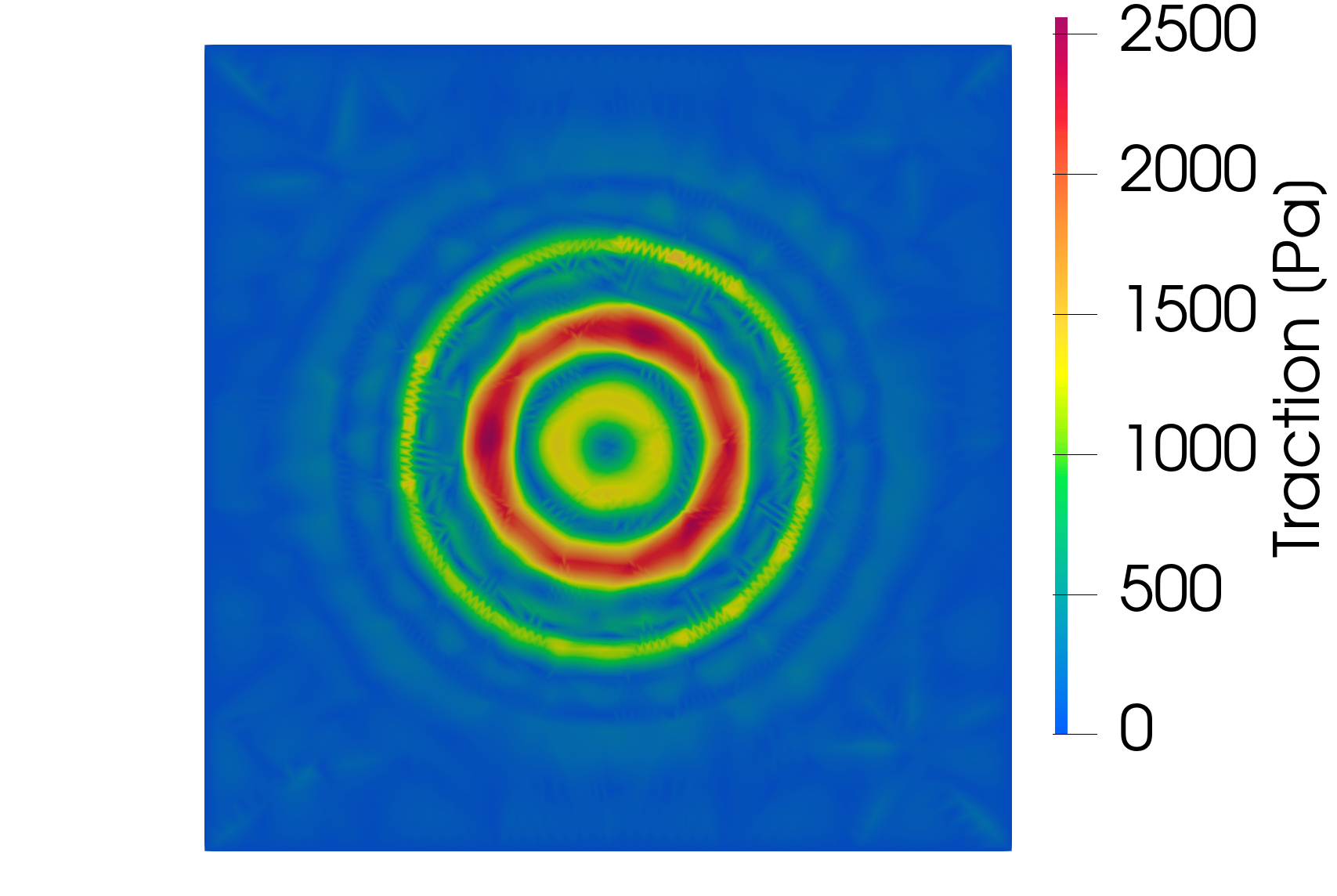}
            \caption{Error of linear reconstruction.}
            \label{fig:comparison_err_lin}
        \end{subfigure}
        \vspace{0.5cm} 
        \begin{subfigure}[t]{\textwidth}
            \centering
            \includegraphics[width=\linewidth]{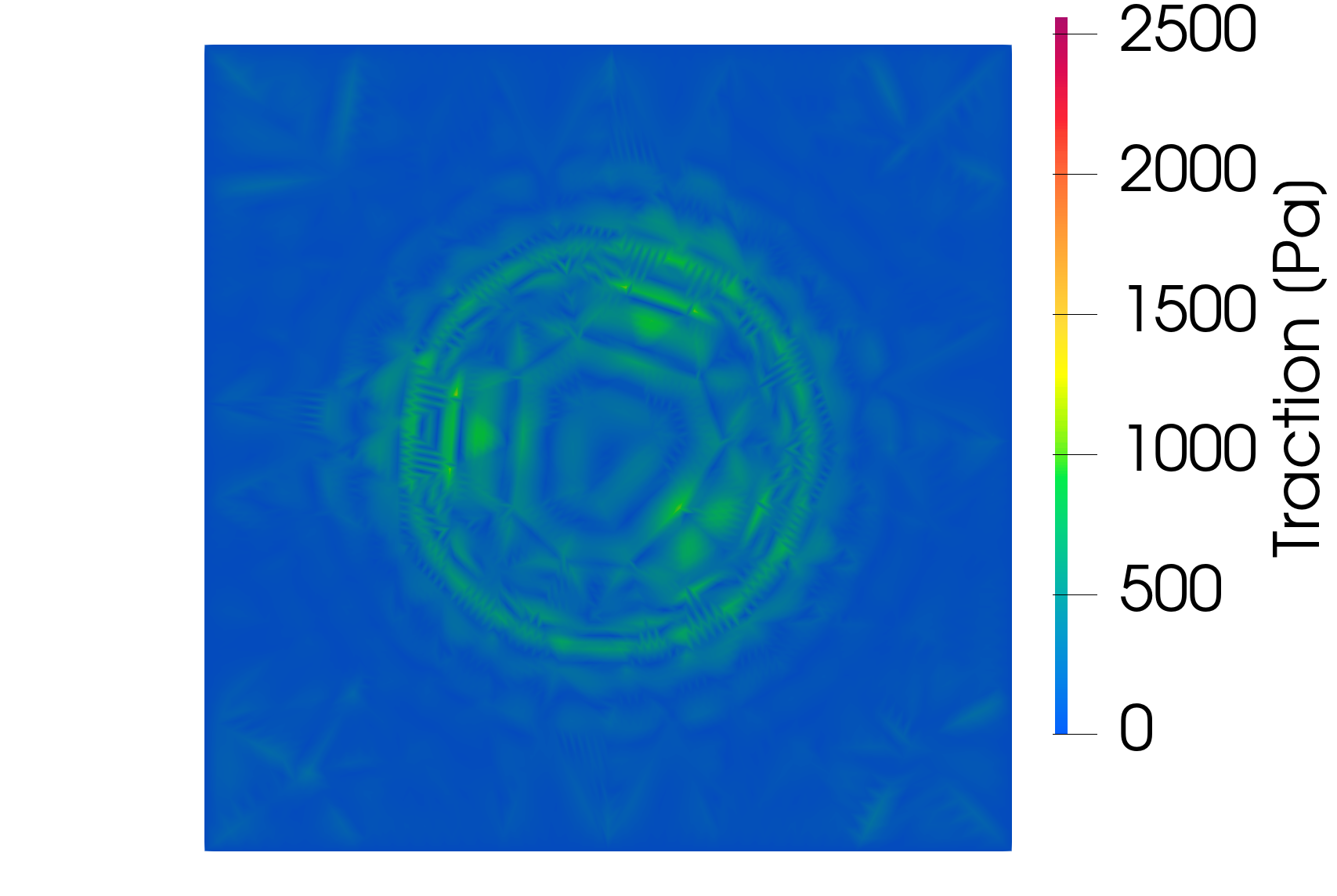}
            \caption{Error of nonlinear reconstruction.}
            \label{fig:comparison_err_nl}
        \end{subfigure}
    \end{minipage}
    \caption{Comparison of nonlinear and linear 2D model, (a) compares the linear and nonlinear forward model, (b) and (c) show the error of the linear/nonlinear reconstruction of force field~\eqref{eq:force1} with $a = 2 \cdot 10^5$ and noise-free data.}
    \label{fig:comparison}
\end{figure}



\section{Conclusion and outlook}

We described a forward model for linear 2.5D TFM and pure linear and nonlinear 2D TFM and analyzed the nonlinear parameter-to-state map $S$, see~\eqref{eq2:tfmForward}. In particular we discussed the well-definedness of the map $S$, showed Fr\'echet-differentiability and calculated its adjoint, which are essential to apply numerical regularization. Furthermore, we derived a suitable nonlinear material law leading to a nonlinear PDE. For this setting, we established a rigorous mathematical analysis for the inverse problem of TFM from biophysics. With the nonlinear approach we developed an inverse method that can be applied to nonlinear materials with high stresses. 
Our approach is developed in a continuous setting, with discretization applied only at the end to preserve the mathematical properties of the involved functions, particularly the solution, as long as possible. Furthermore, our analysis and the inclusion in the Python toolbox regpy~\cite{regpy} allow to easily use different regularization algorithms instead of just using Tikhonov regularization as commonly used in the physics community. In the linear 2D case, where we applied our methods to experimental data, we observe a good coincidence of our results and the ones obtained from the FTTC method. In the nonlinear 2D case and the linear 2.5D case, we obtained conclusive results for synthetic data and were able to reduce artifacts by incorporating the nature of the traction forces.

An advantage of the proposed approach is its flexibility allowing for a number of 
extensions and variations. For example, we could work with other specific nonlinear material laws. Additionally, other a-priori information as, e.g., known support of the traction stress, the equilibrium condition,  symmetry assumptions, or sparsity with respect to certain frames, could be incorporated into the optimization problem. 
Furthermore, the mathematical analysis and implementation of TFM could be extended to 3D TFM.


\vspace*{3ex}

\paragraph{Data availability statement}
The code and data that support the findings of this article are openly available at the following Git repository: \url{https://gitlab.gwdg.de/sarnighausen/traction-force-microscopy} 

\paragraph{Acknowledgements}
The authors thank Emily Klass for help with the implementation. \\
G. Sarnighausen, T.~Nguyen, A.~Wald, T.~Hohage and S.~K\"oster  acknowledge funding by Deutsche Forschungsgemeinschaft (DFG, German Research Foundation) – CRC 1456 (project-ID 432680300), projects A04, B06 and C03. 
A.~Wald, S.~K\"oster and T.~Betz additionally acknowledge support by DFG RTG 2756 (project-ID 449750155), projects A4, A7, B3. S.~Köster,
T.~Betz and U.S.~Schwarz are members of the
Max Planck School Matter to Life supported by the German Federal Ministry of Education and Research (BMBF) in collaboration with the Max Planck Society. 


\vspace*{1ex}

\appendix

\section{Proof of Lemma~\ref{th2:existence Ball}}
\label{Proof_Existence}

We use $\lesssim$, $\gtrsim$ and $\simeq$ for relations up to multiplicative constants independent of $u$ and $v$.\\
\textbf{(i) Preliminaries.}
            To ensure the positivity constraint we define 
        \begin{align*}
            \bar{P}(F,\det F) = 
                \begin{cases}
                     P(F,\det F) & \text{if } \det F > 0 \\
                     +\infty & \text{if } \det F \leq 0
                \end{cases}
        \end{align*}
        with the function $P$ as in Definition~\ref{def:coercivity}.
        The convex function $P$ is continuous according to Proposition 47.5 in \cite{Zeidler1997NonlinearFA4}. Therefore the function
            $
                \bar{P} : \M^2 \times (0, \infty) \to (-\infty, \infty]
            $
        is continuous because of~\eqref{eq:existence,limit}.
        Furthermore, we set 
            \begin{align*}
                X &\coloneqq W^{1,p}(\Omega, \R^2) \times L^s(\Omega,\R^2) \\
                Hu(x) &\coloneqq \big(x + u(x), \det(I + \nabla u(x)\big) \in X
            \end{align*}
            for $u \in W^{1,p}(\Omega, \R^2)$.
            
(a) On the space $W^{1,p}(\Omega,\R^2)$ we introduce the norm
                    \begin{align*}
                         \lVert v \rVert_{1,p} \coloneqq \left( \int_{\Omega} \lvert \nabla v \rvert^p_{\mathrm{F}} \dx + \int_{\partial \Omega} \lvert v \rvert^p_{\mathrm{F}} \ds \right)^{\frac{1}{p}},
                    \end{align*}
                    which is equivalent to the standard norm on the space $W^{1,p}(\Omega,\R^2)$, see e.g.~\cite[\S 114, Theorem 3]{smirnov}.
                    Then, for all elements $v(x) = x + u(x)$ with $u \in U$, we have
                    \begin{align}
                        \label{eq2:proofExistenceNorm}
                        \lVert v \rVert^p_{1,p} \lesssim \lVert \nabla v \rVert^p_{L^p(\Omega,\R^2)} + C_{u_0}\qquad \mbox{with}
                        \qquad C_{u_0}:=\lVert x + u_0 \rVert^p_{L^p(\partial \Omega, \R^2)}
                    \end{align}
                    since for $u \in U$ holds $u = u_0$ on the boundary $\partial \Omega$ and thus the boundary integral has the same value for all elements $u \in U.$ 
                    On the space $X$ we choose the norm 
                    $
                        \lVert v \rVert_X \coloneqq \lVert v_1 \rVert_{1,p} + \lVert v_2 \rVert_s
                    $
                    for $v = (v_1,v_2) \in \X$.
                    
(b) Theorem of Mazur. Let the sequence $(v_n)$ in an arbitrary Banach space $Y$ converge weakly to some element $v \in Y$, i.e., $v_n \rightharpoonup v$ in $Y$ as $n \to \infty$.
                    Then there exist convex linear combinations $w_n = \sum_{i=1}^{N(n)} \lambda_{ni} v_i$ with $\sum_{i=1}^{N(n)} \lambda_{ni} = 1$, $\lambda_{ni} > 0$ such that $N(n) \to \infty$ as $n \to \infty$, and we have strong convergence $w_n \to v$ in the space $Y$ as $n \to \infty$, see, e.g.~\cite[Lemma 3.1.20]{buehlerFunctional}.
                    
(c) Recall that in case of convergence $f_n \to f$ in $L^p(\Omega)$ as $n \to \infty$ and $1 \leq p \leq \infty$, there exists a subsequence -- for simplicity also called $f_n$ -- such that the subsequence converges almost everywhere, i.e., $f_{n}(x) \to f(x)$ a.e.~in $\Omega$ as $n \to \infty$, according to the Riesz-Fischer theorem~\cite[chapter VI,\textsection 2]{elstrodt2018Mass}.

\textbf{(ii) Minimal sequence.} 
                Define the infimum $m \coloneqq \inf_{u \in U} G(u)$. We show that there exists an element $u^* \in U$ with $G(u^*) = m$. Due to the coercivity we have $- \infty < G(u) \leq \infty$ for all elements $u \in U$. Since $G(u_0) < \infty$, the infimum $m$ is finite. 
                
                We choose a sequence $(u_n)$ in the space $U$ such that $G(u_1) \geq G(u_2) \geq \ldots \geq G(u_n) \to m$ as $n \to \infty$ and show that the sequence $(H u_n)$ is bounded in $X$.
                Using \eqref{eq2:proofExistenceNorm} and Definition~\ref{def:coercivity} we obtain
                    \begin{align*}
                         \tilde{H}(u_n) &\hspace{1.1ex}\coloneqq \lVert x + u_n(x) \rVert^p_{1,p} + \lVert \det(I + \nabla u_n) \rVert^s_{L^s(\Omega,\R^2)} \\
                        &\stackrel{\eqref{eq2:proofExistenceNorm}}{\lesssim}\lVert I + \nabla u_n(x)) \rVert^p_{L^p(\Omega,\R^2)} + \lVert \det(I + \nabla u_n) \rVert^s_{L^s(\Omega,\R^2)}  + C_{u_0}\\
                        &\hspace{1.3ex}\simeq  \lVert I + \nabla u_n(x)) \rVert^p_{L^p(\Omega,\R^2)} + \lVert \det(I + \nabla u_n) \rVert^s_{L^s(\Omega,\R^2)}  + C_{u_0} - \frac{1}{p} \lVert I + \nabla u_n(x)) \rVert^p_{L^p(\Omega,\R^2)} \\
                        &\hspace{1.1ex}\stackrel{\ref{def:coercivity}}{\lesssim} \int_\Omega  P(I + \nabla u_n, \det (I + \nabla u_n) \dx - \frac{1}{p} \lVert I + \nabla u_n(x)) \rVert^p_{L^p(\Omega,\R^2)}  + C_{u_0} - D.
                    \end{align*}
                To get rid of the power $p$ we use the Bernoulli inequality
                which states $(1+x)^p \geq 1 + px$ for $x \geq -1$ and $p \geq 1$.
                With the substitution $a = 1 + x$ it follows
                $
                    -\frac{a^p}{p} \leq -a - C_p
                $
                for $a \geq 0$ with 
                $C_p:=(1-p)/p$. Thus
                \begin{align*}
                    \tilde{H}(u_n)
                    \lesssim & \int_\Omega  P(I + \nabla u_n, \det (I + \nabla u_n) \dx - \lVert I + \nabla u_n(x)) \rVert_{L^p(\Omega,\R^2)} - C_p  + C_{u_0} - D.
                \end{align*}

                To relate $\lVert I + \nabla u_n(x) \rVert_{L^p(\Omega,\R^2)}$ and $\lVert u_n(x) \rVert_{1,p}$ we observe
                \begin{align}
                    \label{eq2:proofExistenceRelation}
                    \begin{split}
                    \lVert u_n(x) \rVert_{1,p}
                    &= \lVert I + \nabla u_n(x) - I \rVert_{L^p(\Omega,\R^2)} + \lVert u_0\rVert_{L^p(\partial \Omega, \R^2)} \\
                    &\leq \lVert I + \nabla u_n(x) \rVert_{L^p(\Omega,\R^2)} + \lVert u_0\rVert_{L^p(\partial \Omega, \R^2)} + \lVert I \rVert_{L^p(\Omega,\R^2)}
                    \end{split}
                \end{align}
                \vspace{-2.5ex}
                and with $C_{u'_0}:=\lVert u_0\rVert_{L^p(\partial \Omega, \R^2)} + \lVert I \rVert_{L^p(\Omega,\R^2)}$
                we get
                \begin{align*}
                    \tilde{H}(u_n)
                    \stackrel{\eqref{eq2:proofExistenceRelation}}{\lesssim} & \int_\Omega  P(I + \nabla u_n, \det (I + \nabla u_n) \dx - \lVert u_n(x) \rVert_{1,p} + C_{u'_0} - C_p  + C_{u_0} - D.
                \end{align*}
                Next we estimate the $W^{1,p}$-norm by the $L^{\frac{p}{p-1}}$-norm.
                The embedding $W^{1,p}(\Omega,\R^2) \hookrightarrow L^{\frac{p}{p-1}}(\Omega,\R^2)$ is compact and thus continuous due to the Rellich-Kondrachov theorem
                \begin{align}
				\label{eq2:proofExistenceEmbedding}
				    W^{k_1,p_1}(\Omega,\R^2) \hookrightarrow W^{k_2,p_2}(\Omega,\R^2), 
                \quad \text{if  } \frac{k_1 - k_2}{2} > \frac{1}{p_1} - \frac{1}{p_2}, k_1 \geq k_2,
				\end{align} 
                for $k_1, k_2 \in \N_0, p_1, p_2 \geq 1, $ see, e.g.,~\cite[Chapter 6]{adams2003sobolev}, which is fulfilled for $k_1 = 1$, $k_2 = 0$, $p_1 = p$ and $p_2 = \frac{p}{p-1}$ because $p \geq 2$. Finally we get
                \begin{align*}
                    \tilde{H}(u_n)                         
                    &\stackrel{\eqref{eq2:proofExistenceEmbedding}}{\lesssim} \int_\Omega W(I + \nabla u_n) \dx - \lVert T \rVert_{L^p(\Omega,\R^2)} \lVert u_n(x) \rVert_{L^{\frac{p}{p-1}}(\Omega,\R^2)} + C_{u'_0} - C_p  + C_{u_0} - D \\
                       &\hspace{1.5ex} \lesssim \int_\Omega W(I + \nabla u_n) \dx - \int_\Omega T u_n \dx + C_{u'_0} - C_p  + C_{u_0} - D\\
                       &\hspace{1.5ex} = G(u_n) + C_{u'_0} - C_p  + C_{u_0} - D
                \end{align*}
                This is bounded because all the constants are bounded and $G(u_n) \leq G(u_1)$ for all $n \in \N$.
                Now since $\lVert x + u_n(x) \rVert^p_{1,p} + \lVert \det(I + \nabla u_n) \rVert^s_{L^s(\Omega,\R^2)}$ is bounded, the expression without powers $\lVert H u_n \rVert_X$ is bounded as well.

\textbf{(iii) Weak and strong convergence.} 
                Since the space $X$ is reflexive, there exists a subsequence, denoted again by $(H u_n)$ that converges weakly to an element $v \in X$, i.e., $Hu_n \rightharpoonup v$ in $X$ as $n \to  \infty$. 
                Theorem 6.2 in \cite{Ball1976ConvexityCA} implies $v = Hu^*$ for an element $u^* \in U$.
                
                For each element $v_n = Hu_n$, there exists a convex linear combination $w_n$ of $H u_i$ from Mazur's theorem as in (i)b). By possibly passing to a subsequence, it holds $w_n \to Hu^*$ in $X$ and $w_n(x) \to Hu^*(x)$ a.e.~in $\Omega$ as $n \to \infty$.

\textbf{(iv) Lemma of Fatou.}
                The convexity of the function $P$ yields
                $
                    P(w_n) \leq \sum_{i=1}^{N(n)} \lambda_{ni} P(Hu_i).
                $
                From the coercivity condition follows $P(w_n) \geq D$ for all $n$ with the constant $D$ from~\eqref{def:coercivity}, and the Lemma of Fatou (see e.g. \cite[Appendix]{Zeidler1997NonlinearFA4}) and pointwise convergence of the convex linear combinations $w_n$ from (iii) yield
                \begin{align*}
                    \int_{\Omega} \bar{P}(u^*) \dx = \int_{\Omega} \lim_{n \to \infty} P(w_n) \dx \leq \lim_{n \to \infty} \int_{\Omega} \sum_{i = 1}^{N(n)} \lambda_{ni} P(H u_i) \dx.
                \end{align*}
                We define $\bar{G}$ by $\bar{G}(u) = \int_\Omega P(F,\det F) - Tu \dx$. Then we get 
                \begin{align}
                \label{eq:convexComb}
                    \bar{G}(u^*) = \int_\Omega \bar{P}(u^*) \dx - \int_\Omega T u^* \dx 
                     \leq \lim_{n \to \infty} \sum_{i=1}^{N(n)} \lambda_{ni} G(u_i) \stackrel{(1)}{=} \lim_{n \to \infty} G(u_{N(n)}) = m
                \end{align}
                where equality (1) holds because the expressions $G(u_{N(n)})$ and $\sum_{i=1}^{N(n)} \lambda_{ni} G(u_i)$ have the same limit due to the condition $\sum_{i=1}^{N(n)} \lambda_{ni} = 1$, $\lambda_{ni} > 0$. 
                So the limit point $u^*$ is a minimizer of the energy $G$.
                
\textbf{(v)} It remains to show that the minimizer $u^*$ is an element of the space $U$.
                Due to the condition $\bar{P}(u^*) < \infty$, 
                the construction of the function $\bar{P}$ yields $\det(I + \nabla u^*(x)) > 0$ a.e.~in $\Omega$. 
                Since $u_n = u_0$ on the boundary $\partial \Omega$ for all $n$, 
                and since convex linear combinations of the elements $u_n$ converge to the minimizer $u^* \in W^{1,p}(\Omega, \R^2)$, 
                we also have $u^* = u_0$ on the boundary $\partial \Omega$. 
                Thus \eqref{eq:convexComb} is equivalent to $G(u^*) = m$ and $u^* \in U$.  


\bibliographystyle{plain}
\bibliography{lit}

\pagestyle{myheadings}
\thispagestyle{plain}

\end{document}